\documentclass[a4paper,12pt]{amsart}

\usepackage{amsfonts}
\usepackage{amsmath}
\usepackage{amssymb}
\usepackage{graphicx}

\usepackage[usenames]{color}
\usepackage[colorlinks]{hyperref}

\newtheorem{thm}{Theorem}[section]
\newtheorem{lem}[thm]{Lemma}

\newtheorem{assum}[thm]{Assumption}

\theoremstyle{definition}

\theoremstyle{remark}
\newtheorem{rem}{Remark}[section]
\newtheorem{defn}{Definition}

\numberwithin{equation}{section}

\DeclareMathOperator{\supp}{supp}

\begin{document}

\title[Tsunami propagation for singular topographies]{Tsunami propagation for singular topographies}

\author[A. Altybay]{Arshyn Altybay}
\address{
  Arshyn Altybay:
  \endgraf
  Al-Farabi Kazakh National University
  \endgraf
  Almaty, Kazakhstan
  \endgraf
   and  
  \endgraf
  Department of Mathematics: Analysis, Logic and Discrete Mathematics
  \endgraf
  Ghent University, Belgium
  \endgraf
   and
  \endgraf
   Institute of Mathematics and Mathematical Modeling
  \endgraf
  Almaty, Kazakhstan
  \endgraf
  {\it E-mail address} {\rm arshyn.altybay@gmail.com}
 }

\author[M. Ruzhansky]{Michael Ruzhansky}
\address{
	Michael Ruzhansky:
	 \endgraf
  Department of Mathematics: Analysis, Logic and Discrete Mathematics
  \endgraf
  Ghent University, Belgium
  \endgraf
  and
  \endgraf
  School of Mathematical Sciences
    \endgraf
    Queen Mary University of London
  \endgraf
  United Kingdom
	\endgraf
  {\it E-mail address} {\rm michael.ruzhansky@ugent.be}
}

\author[M. Sebih]{Mohammed Elamine Sebih}
\address{
  Mohammed Elamine Sebih:
  \endgraf
  Laboratory of Analysis and Control of Partial Differential Equations
  \endgraf
  Djillali Liabes University
  \endgraf
  Sidi Bel Abbes, Algeria
  \endgraf
   and
  \endgraf
  Department of Mathematics: Analysis, Logic and Discrete Mathematics
  \endgraf
  Ghent University, Belgium
  \endgraf
  {\it E-mail address} {\rm sebihmed@gmail.com}
 }

\author[N. Tokmagambetov]{Niyaz Tokmagambetov}
\address{
  Niyaz Tokmagambetov:
    \endgraf
  Department of Mathematics: Analysis, Logic and Discrete Mathematics
  \endgraf
  Ghent University, Belgium
  \endgraf
   and
  \endgraf
  Al-Farabi Kazakh National University
  \endgraf
  Almaty, Kazakhstan
  \endgraf
  {\it E-mail address} {\rm niyaz.tokmagambetov@ugent.be and tokmagambetov@math.kz}
 }

\thanks{The authors were supported  by the FWO Odysseus 1 grant G.0H94.18N: Analysis and Partial Differential Equations. MR was supported in parts by the EPSRC Grant
EP/R003025/1, by the Leverhulme Research Grant RPG-2017-151.}

\keywords{Tsunami propagation, wave equation, Cauchy problem, weak solution, singular coefficient, regularisation, very weak solution, numerical analysis.}
\subjclass[2010]{35L81, 35L05, 	35D30, 35A35.}

\begin{abstract}
We consider a tsunami wave equation with singular coefficients and prove that it has a very weak solution. Moreover, we show the uniqueness results and consistency theorem of the very weak solution with the classical one in some appropriate sense. Numerical experiments are done for the families of regularised problems in one- and two-dimensional cases. In particular, the appearance of a substantial second wave is observed, travelling in the opposite direction from the point/line of singularity. Its structure and strength are analysed numerically. In addition, for the two-dimensional tsunami wave equation, we develop GPU computing algorithms to reduce the computational cost.
\end{abstract}

\maketitle

\tableofcontents


\section{Introduction}
In this work we consider the Cauchy problem for the tsunami wave equation governed by the shallow water equations. Namely, for $T>0$, we study the Cauchy problem
\begin{equation}
    \left\lbrace
    \begin{array}{l}
    u_{tt}(t,x) - \sum_{j=1}^{d}\partial_{x_{j}}\left(h_{j}(x)\partial_{x_{j}}u(t,x)\right)=0, \,\,\,(t,x)\in\left[0,T\right]\times \mathbb{R}^{d},\\
    u(0,x)=u_{0}(x), \,\,\, u_{t}(0,x)=u_{1}(x), \,\,\, x\in\mathbb{R}^{d}, \label{Equation introduction}
    \end{array}
    \right.
\end{equation}
where $\boldsymbol{h}:\mathbb{R}^d \rightarrow \mathbb{R}^d,$ $x\mapsto {\boldsymbol{h}}(x)=\left(h_{1}(x),...,h_{d}(x)\right)^{T}$ is a vector valued function. Our model is a general case of a well known physical model when $h=h_j, \,\, j=1, \dots, d,$ is real valued. In this particular case, $h$ denotes the water depth and $u$ represents the free surface displacement. Let us start by the description of the physical motivation.

Tsunamis are a series of traveling waves in water induced by the displacement of the sea floor due to earthquakes or landslides. Three stages of tsunami development are usually distinguished: the generation phase, the propagation of the waves in the open ocean (or sea) and the propagation near the shoreline. Since the wavelengths of tsunamis are much greater than the water depth, they are often modelled using the shallow water equations. The most common model used to describe tsunamis (see, for instance \cite{Kund07}, \cite{DD07}, \cite{Ren17}, \cite{RS08}, \cite{RS10}, \cite{DDORS14}, \cite{ADD19} and the references therein) is
\begin{equation}
    \left\lbrace
    \begin{array}{l}
    u_{tt}(t,x) - \sum_{j=1}^{d}\partial_{x_{j}}\left(h(x)\partial_{x_{j}}u(t,x)\right)=f(t,x), \,\,\,(t,x)\in\left[0,T\right]\times \mathbb{R}^{d},\\
    u(0,x)=0, \,\,\, u_{t}(0,x)=0, \,\,\, x\in\mathbb{R}^{d}, \label{Equation forced}
    \end{array}
    \right.
\end{equation}
where $f(t,x)$ is a source term related to the formation of a localized disturbance in the first stage of the tsunami life. When analysing the system at the final stages, that is for $t\geq t_0>0$, the source term can be neglected and a homogeneous equation can be considered instead:
\begin{equation}
    u_{tt}(t,x) - \sum_{j=1}^{d}\partial_{x_{j}}\left(h(x)\partial_{x_{j}}u(t,x)\right)=0, \,\,\,(t,x)\in\left[t_0,T\right]\times \mathbb{R}^{d}, \label{Equation free}
\end{equation}
where the initial free surface displacement and the initial velocity can be described by known functions of the spacial variable, i.e.
\begin{equation}
     u(t_0,x)=u_{0}(x), \,\,\, u_{t}(t_0,x)=u_{1}(x), \,\,\, x\in\mathbb{R}^{d}.
\end{equation}
In the present paper, we are interested in the final stages of the tsunami development. So, we consider the latter model, and for the sake of simplicity we take $0$ as the initial time instead of $t_0$. That is, we consider
\begin{equation}
    \left\lbrace
    \begin{array}{l}
    u_{tt}(t,x) - \sum_{j=1}^{d}\partial_{x_{j}}\left(h(x)\partial_{x_{j}}u(t,x)\right)=0, \,\,\,(t,x)\in\left[0,T\right]\times \mathbb{R}^{d},\\
    u(0,x)=u_{0}(x), \,\,\, u_{t}(0,x)=u_{1}(x), \,\,\, x\in\mathbb{R}^{d}, \label{Equation introduction1}
    \end{array}
    \right.
\end{equation}
where we allow the water depth coefficient $h$ to be discontinuous or even to have less regularity. The singularity of $h$ can be interpreted as sudden changes in the water depth caused by the interaction of the wave with complicated topographies of the sea floor such as bays and harbours.

While from a physical point of view this is a natural setting, mathematically we face a problem: If we are looking for distributional solutions, the term $h(x)\partial_{x_i}u(t,x)$ does not make sense in view of Schwartz famous impossibility result about multiplication of distributions \cite{Sch54}. In this context the concept of very weak solutions was introduced in \cite{GR15}, for the analysis of second order hyperbolic equations with irregular coefficients and was further applied in a series of papers \cite{ART19}, \cite{RT17a} and \cite{RT17b} for different physical models, in order to show a wide applicability. In \cite{MRT19, SW20} it was applied for a damped wave equation with irregular dissipation arising from acoustic problems and an interesting phenomenon of the reflection of the original propagating wave was numerically observed. In all these papers the theory of very weak solutions is dealt for time-dependent equations. In the recent works \cite{Gar20, ARST20a, ARST20b, ARST20c}, the authors start to study the concept of very weak solutions for partial differential equations with space-depending coefficients.

It is shown there, that this notion is very well adapted for numerical simulations when a rigorous mathematical formulation of the problem is difficult in the framework of the classical theory of distributions. Furthermore, by the theory of very weak solutions we can talk about uniqueness of numerical solutions to differential equations. So, here we consider the Cauchy problem (\ref{Equation introduction1}) and prove that it has a very weak solution.

Moreover, since numerical solutions are useful for predicting and understanding tsunami propagation, many numerical models are developed in the literature, we cite fore instance \cite{Behr10, LGB11, RHH11, BD15}. As a second task in the present paper we do some numerical computations, where we observe interesting behaviours of solutions.

\section{Main results}
For $T>0$, we consider the Cauchy problem
\begin{equation}
    \left\lbrace
    \begin{array}{l}
    u_{tt}(t,x) - \sum_{j=1}^{d}\partial_{x_{j}}\left(h_{j}(x)\partial_{x_{j}}u(t,x)\right)=0, \,\,\,(t,x)\in\left[0,T\right]\times \mathbb{R}^{d},\\
    u(0,x)=u_{0}(x), \,\,\, u_{t}(0,x)=u_{1}(x), \,\,\, x\in\mathbb{R}^{d}, \label{Equation}
    \end{array}
    \right.
\end{equation}
where ${\boldsymbol{h}}:\mathbb{R}^{d}\rightarrow \mathbb{R}^{d}$; $x\mapsto {\boldsymbol{h}}(x)=\left(h_{1}(x),...,h_{d}(x)\right)^{T}$ is singular and positive in the sense that there exists $c_0 >0$ such that for all $j=1,...,d$ we have $0<c_0 \leq h_j$. The following lemma is a key of the proof of existence, uniqueness and consistency of a very weak solution to our model problem. It is stated in the case when ${\boldsymbol{h}}$ is a regular vector-function.

In what follows we will use the following notations. By writing $a\lesssim b$ for functions $a$ and $b$, we mean that there exists a positive constant $c$ such that $a\leq c b$. Also, we denote
\begin{equation*}
\Vert u(t,\cdot)\Vert := \Vert u(t,\cdot)\Vert_{H^{2}} = \Vert u(t,\cdot)\Vert_{L^2} + \Vert\sum_{j=1}^{d} \partial_{x_j} u(t,\cdot)\Vert_{L^2} + \Vert \Delta u(t,\cdot)\Vert_{L^2}.
\end{equation*}
In addition, we introduce the Sobolev space $W^{1,\infty}(\mathbb{R}^d)$ by
\begin{equation*}
W^{1,\infty}(\mathbb{R}^d):=\left\{ f \text{\,is measurable:}\, \Vert f\Vert_{W^{1,\infty}}:= \Vert f\Vert_{L^{\infty}} + \Vert \nabla f\Vert_{L^{\infty}} < +\infty \right\}.
\end{equation*}

\begin{thm}
\label{lem 1}
Let ${\boldsymbol{h}}\in \left[L^{\infty}\left(\mathbb{R}^d\right)\right]^{d}$ be positive. Assume that $u_0 \in H^{1}(\mathbb{R}^d)$ and $u_1 \in L^{2}(\mathbb{R}^d)$.
Then, the unique solution $u\in C([0, T]; H^{1}(\mathbb R^d))\cap C^{1}([0, T]; L^{2}(\mathbb R^d))$ to the Cauchy problem (\ref{Equation}), satisfies the estimates
\begin{equation}
    \Vert u(t,\cdot)\Vert_{L^2} + \Vert u_{t}(t,\cdot)\Vert_{L^2} + \sum_{j=1}^{d}\Vert \partial_{x_j}u(t,\cdot)\Vert_{L^2} \lesssim \left(1 + \sum_{j=1}^{d}\Vert h_j\Vert_{L^{\infty}}^{\frac{1}{2}}\right) \left[ \Vert u_0\Vert_{H^1} + \Vert u_1\Vert_{L^2} \right], \label{Energy estimate}
\end{equation}
for all $t\in [0,T]$.

In addition, assume that $\boldsymbol{h}\in \left[W^{1,\infty}(\mathbb{R}^d)\right]^d$, $u_0 \in H^{2}(\mathbb{R}^d)$ and $u_1 \in H^{1}(\mathbb{R}^d)$. Moreover, if $h_j(x) = h(x)$ for all $j=1, \dots, d$. Then, the solution $u\in C([0, T]; H^{2}(\mathbb R^d))\cap C^{1}([0, T]; H^{1}(\mathbb R^d))$ satisfies the estimate
\begin{equation}
\Vert \Delta u\Vert_{L^2} \lesssim H\left(1+H\right)\left[\Vert u_0\Vert_{H^2} + \Vert u_1\Vert_{H^1}\right], \label{Energy estimate 1}
\end{equation}
for all $t\in [0,T]$, where $H = \max \left\{ \Vert h\Vert_{W^{1,\infty}}^{\frac{1}{2}},\,\Vert h\Vert_{W^{1,\infty}} \right\}$.
\end{thm}

\begin{proof}
We multiply the equation in (\ref{Equation}) by $u_t$ and we integrate with respect to the variable $x$, to obtain
\begin{equation}
Re \left( \langle u_{tt}(t,\cdot),u_{t}(t,\cdot)\rangle_{L^2} + \sum_{j=1}^{d}\langle i\partial_{x_j} (h_{j}i\partial_{x_j}u(t,\cdot)),u_{t}(t,\cdot)\rangle_{L^2} \right) = 0, \label{Energy functional}
\end{equation}
where $\langle \cdot, \cdot \rangle_{L^2}$ denotes the inner product of the Hilbert space $L^2(\mathbb{R}^{d})$ and $i$ is the imaginary unit, such that $i^2 =-1$.
After short calculations, we easily show that
\begin{equation}
    Re \langle u_{tt}(t,\cdot),u_{t}(t,\cdot)\rangle_{L^2} = \frac{1}{2}\partial_{t} \Vert u_t\Vert_{L^2}^2
\end{equation}
and
\begin{equation}
    Re \sum_{j=1}^{d}\langle i\partial_{x_j} (h_{j}i\partial_{x_j}u(t,\cdot)),u_{t}(t,\cdot)\rangle_{L^2} = \frac{1}{2}\sum_{j=1}^{d}\partial_{t} \Vert h_{j}^{\frac{1}{2}}\partial_{x_j}u(t,\cdot)\Vert_{L^2}^2.
\end{equation}
Then, from (\ref{Energy functional}), we get the energy conservation formula
\begin{equation}
\label{CL-01}
\partial_{t} \left( \Vert u_t\Vert_{L^2}^2 + \sum_{j=1}^{d} \Vert h_{j}^{\frac{1}{2}}\partial_{x_j}u(t,\cdot)\Vert_{L^2}^2 \right) = 0.
\end{equation}
By taking in consideration that $\Vert h_{j}^{\frac{1}{2}}\partial_{x_j}u_0\Vert_{L^2}^2$ can be estimated by
\begin{equation}
    \Vert h_{j}^{\frac{1}{2}}\partial_{x_j}u_0\Vert_{L^2}^2 \leq \Vert h_j\Vert_{L^{\infty}}\Vert u_0\Vert_{H^1}^2
\end{equation}
for all $j=1,...,d$, it follows that
\begin{equation}
    \Vert u_t\Vert_{L^2}^2 \leq \Vert u_1\Vert_{L^2}^2 + \sum_{j=1}^{d}\Vert h_j\Vert_{L^{\infty}}\Vert u_0\Vert_{H^1}^2 \label{Estimate u_t}
\end{equation}
and
\begin{equation}
    \Vert h_{i}^{\frac{1}{2}}\partial_{x_i}u(t,\cdot)\Vert_{L^2}^2 \leq \Vert u_1\Vert_{L^2}^2 + \sum_{j=1}^{d}\Vert h_j\Vert_{L^{\infty}}\Vert u_0\Vert_{H^1}^2,
\label{Estimate h u}
\end{equation}
for all $i=1,...,d$.

In the last inequality, using that the left hand side can be estimated by
\begin{equation}
    \Vert h_{i}^{\frac{1}{2}}\partial_{x_i}u(t,\cdot)\Vert_{L^2}^2 \geq \inf\limits_{x\in\mathbb R^{d}}|h_i(x)| \, \Vert \partial_{x_i}u(t,\cdot)\Vert_{L^2}^2,
\end{equation}
and that $h$ is positive, we get for all $i=1,...,d$ the estimate
\begin{equation}
    \Vert \partial_{x_i}u(t,\cdot)\Vert_{L^2}^2 \lesssim \Vert u_1\Vert_{L^2}^2 + \sum_{j=1}^{d}\Vert h_j\Vert_{L^{\infty}}\Vert u_0\Vert_{H^1}^2. \label{Estimate u_x_j}
\end{equation}
Let us estimate $u$. By the fundamental theorem of calculus we have that
\begin{equation}
    u(t,x) = u_0(x) + \int_{0}^{t} u_{t}(s,x) ds. \label{Representation u}
\end{equation}
Taking the $L^2$ norm in (\ref{Representation u}) and using (\ref{Estimate u_t}) to estimate $u_t$, we arrive at
\begin{equation}
    \Vert u(t,\cdot)\Vert_{L^2} \lesssim \left(1 + \sum_{j=1}^{d}\Vert h_j\Vert_{L^{\infty}}^{\frac{1}{2}}\right) \left[ \Vert u_0\Vert_{H^1} + \Vert u_1\Vert_{L^2} \right]. \label{Estimate u}
\end{equation}


Now, let us assume that $\boldsymbol{h}\in \left[W^{1,\infty}(\mathbb{R}^d)\right]^d$, $u_0 \in H^{2}(\mathbb{R}^d)$ and $u_1 \in H^{1}(\mathbb{R}^d)$. We note that, if $u$ solves the Cauchy problem
\begin{equation}
    \left\lbrace
    \begin{array}{l}
    \partial_{t}^{2}u(t,x) - \sum_{j=1}^{d}\partial_{x_{j}}\left(h_{j}(x)\partial_{x_{j}}u(t,x)\right)=0, \,\,\,(t,x)\in\left[0,T\right]\times \mathbb{R}^{d},\\
    u(0,x)=u_{0}(x), \,\,\, u_{t}(0,x)=u_{1}(x), \,\,\, x\in\mathbb{R}^{d},
    \end{array}
    \right.
\end{equation}
then $u_t$ solves
\begin{equation}
    \left\lbrace
    \begin{array}{l}
    \partial_{t}^{2}u_{t}(t,x) - \sum_{j=1}^{d}\partial_{x_{j}}\left(h_{j}(x)\partial_{x_{j}}u_{t}(t,x)\right)=0, \,\,\,(t,x)\in\left[0,T\right]\times \mathbb{R}^{d},\\
    u_{t}(0,x)=u_{1}(x), \,\,\, \partial_{t}u_{t}(0,x)=\sum_{j=1}^{d}\partial_{x_{j}}\left(h_{j}(x)\partial_{x_{j}}u_{0}(x)\right), \,\,\, x\in\mathbb{R}^{d}.
    \end{array}
    \right.
\end{equation}
Then, using the estimates (\ref{Estimate u_t}) and \eqref{Estimate h u}, we get
\begin{equation}
    \Vert u_{tt}(t,\cdot)\Vert_{L^2} \lesssim \sum_{j=1}^{d}\Vert h_j\Vert_{W^{1,\infty}} \Vert u_0\Vert_{H^2} + \sum_{j=1}^{d}\Vert h_j\Vert_{W^{1,\infty}}^{\frac{1}{2}} \Vert u_1\Vert_{H^1},
\label{Estimate u_tt}
\end{equation}
\begin{equation}
    \Vert h_{i}^{\frac{1}{2}}\partial_{x_i}u_{t}(t,\cdot)\Vert_{L^2} \lesssim \sum_{j=1}^{d}\Vert h_j\Vert_{W^{1,\infty}} \Vert u_0\Vert_{H^2} + \sum_{j=1}^{d}\Vert h_j\Vert_{W^{1,\infty}}^{\frac{1}{2}} \Vert u_1\Vert_{H^1},
\label{Estimate u_t_x}
\end{equation}
where for all $i=1,...,d$, we estimated $\Vert \partial_{x_{i}}\left(h_{i}(\cdot)\partial_{x_{i}}u_{0}(\cdot)\right)\Vert_{L^2}$ by
\begin{equation}
    \Vert \partial_{x_{i}}\left(h_{i}(\cdot)\partial_{x_{i}}u_{0}(\cdot)\right)\Vert_{L^2} \lesssim \Vert h_i\Vert_{W^{1,\infty}} \Vert u_0\Vert_{H^2}.
\end{equation}

To get the estimate \eqref{Energy estimate 1}, we need the following result.
\begin{lem} Assume that $h_{j}(x)=h(x)$ for all $j=1, \dots, d.$ Under the conditions and arguments of Theorem \ref{lem 1}, we obtain
$$
\Vert \Delta u(t,\cdot)\Vert_{L^2}^{2} \lesssim \Vert h(\cdot) \sum_{j=1}^{d} \partial_{x_j}^{2}u(t,\cdot)\Vert_{L^2}^{2} = \Vert \sum_{j=1}^{d}  h_{j}(\cdot)\partial_{x_j}^{2}u(t,\cdot)\Vert_{L^2}^{2},
$$
for all $t\in[0, T]$.
\label{L1}
\end{lem}
\begin{proof} Using the assumption that $h_i$ are bounded from below, that is,
$$
\min\limits_{0\leq i \leq d}\{\inf\limits_{x\in\mathbb R^d}h_{i}(x)\}=c_0>0,
$$
for all $i=1,...,d$, we get
$$
\Vert \Delta u(t,x)\Vert_{L^2}^{2} \lesssim c_0^2 \, \Vert \sum_{j=1}^{d} \partial_{x_j}^{2}u(t,x)\Vert_{L^2}^{2} \leq \Vert h(x) \sum_{j=1}^{d} \partial_{x_j}^{2}u(t,x)\Vert_{L^2}^{2}.
$$
It proves the lemma.
\end{proof}

The equation in (\ref{Equation}) implies
\begin{equation}
    \sum_{j=1}^{d}h_{j}(x)\partial_{x_j}^{2}u(t,x) = u_{tt}(t,x) - \sum_{j=1}^{d}\partial_{x_j}h_{j}(x)\partial_{x_j}u(t,x). \label{Equation Delta u}
\end{equation}
Taking the $L^2$-norm on both sides in (\ref{Equation Delta u}) and using Lemma \ref{L1}, we obtain
\begin{align}
    \Vert \Delta u(t,x)\Vert_{L^2} & \lesssim \Vert u_{tt}(t,\cdot)\Vert_{L^2} + \sum_{j=1}^{d}\Vert \partial_{x_j}h_{j}(\cdot)\partial_{x_j}u(t,\cdot)\Vert_{L^2} \nonumber \\
    & \lesssim \Vert u_{tt}(t,\cdot)\Vert_{L^2} + \sum_{j=1}^{d}\Vert h_{j}(\cdot)\Vert_{W^{1,\infty}}\Vert \partial_{x_j}u(t,\cdot)\Vert_{L^2}.
\end{align}
Using so far proved estimates (\ref{Estimate u_x_j}) and (\ref{Estimate u_tt}), we get our estimate for $\Delta u$. This ends the proof of the theorem.
\end{proof}

\subsection{Existence of a very weak solution}
In what follows, we consider the Cauchy problem
\begin{equation}
    \left\lbrace
    \begin{array}{l}
    u_{tt}(t,x) - \sum_{j=1}^{d}\partial_{x_{j}}\left(h_{j}(x)\partial_{x_{j}}u(t,x)\right)=0, \,\,\,(t,x)\in\left[0,T\right]\times \mathbb{R}^{d},\\
    u(0,x)=u_{0}(x), \,\,\, u_{t}(0,x)=u_{1}(x), \,\,\, x\in\mathbb{R}^{d}, \label{Equation singular}
    \end{array}
    \right.
\end{equation}
with singular coefficients and initial data. Now we want to prove that it has a very weak solution. To start with, we regularise the coefficients $h_i$ and the Cauchy data $u_0$ and $u_1$ by convolution with a suitable mollifier $\psi$, generating families of smooth functions $(h_{i,\varepsilon})_{\varepsilon}$, $(u_{0,\varepsilon})_{\varepsilon}$ and $(u_{1,\varepsilon})_{\varepsilon}$, that is
\begin{equation}
    h_{i,\varepsilon}(x) = h_i\ast \psi_{\varepsilon}(x) \,\,\, \text{for}\,\, i=1,...,d
\end{equation}
and
\begin{equation}
    u_{0,\varepsilon}(x) = u_0\ast \psi_{\varepsilon}(x), \,\,\, u_{1,\varepsilon}(x) = u_1\ast \psi_{\varepsilon}(x),
\end{equation}
where
\begin{equation}
    \psi_{\varepsilon}(x) = \varepsilon^{-1}\psi(x/\varepsilon),\,\,\,\varepsilon\in\left(0,1\right].
\end{equation}
The function $\psi$ is a Friedrichs-mollifier, i.e. $\psi\in C_{0}^{\infty}(\mathbb{R}^{d})$, $\psi\geq 0$ and $\int\psi =1$.

\begin{assum}
In order to prove the well posedness of the Cauchy problem (\ref{Equation singular}) in the very weak sense, we ask for the regularisations of the coefficients $(h_{i,\varepsilon})_{\varepsilon}$ and the Cauchy data $(u_{0,\varepsilon})_{\varepsilon}$, $(u_{1,\varepsilon})_{\varepsilon}$ to satisfy the assumptions that there exist $N_0,N_1,N_2\in \mathbb{N}_0$ such that
\begin{equation}
    \Vert h_{i,\varepsilon}\Vert_{W^{1,\infty}} \lesssim \varepsilon^{-N_0} \label{assum coeff}
\end{equation}
for $i=1,...,d$ and
\begin{equation}
    \Vert u_{0,\varepsilon}\Vert_{H^{2}} \lesssim \varepsilon^{-N_1},\,\,\,\,\,\Vert u_{1,\varepsilon}\Vert_{H^{1}} \lesssim \varepsilon^{-N_2}. \label{assum data}
\end{equation}
\end{assum}

\begin{rem}
We note that making an assumption on the regularisation is more general than making it on the function itself. We also mention that such assumptions on distributional coefficients, are natural. Indeed, we know that for $T\in \mathcal{E}'(\mathbb{R}^{d})$ we can find $n\in \mathbb{N}$ and functions $f_{\alpha}\in C(\mathbb{R}^{d})$ such that, $T=\sum_{\vert \alpha\vert \leq n}\partial^{\alpha}f_{\alpha}$. The convolution of $T$ with a mollifier gives
\begin{equation}
    T\ast\psi_{\varepsilon}=\sum_{\vert \alpha\vert \leq n}\partial^{\alpha}f_{\alpha}\ast\psi_{\varepsilon}=\sum_{\vert \alpha\vert \leq n}f_{\alpha}\ast\partial^{\alpha}\psi_{\varepsilon}=\sum_{\vert \alpha\vert \leq n}\varepsilon^{-\vert\alpha\vert}f_{\alpha}\ast\left(\varepsilon^{-1}\partial^{\alpha}\psi(x/\varepsilon)\right),
\end{equation}
and we easily see that the regularisation of $T$ satisfy the above assumption. Fore more details, we refer to the structure theorems for distributions (see, e.g. \cite{FJ98}).
\end{rem}

\begin{defn}[Moderateness] \label{defn:Moderatness}
\leavevmode
\begin{itemize}
    \item[(i)]  A net of functions $(f_{\varepsilon})_{\varepsilon}$, is said to be $H^1$-moderate, if there exist $N\in\mathbb{N}_{0}$ such that
\begin{equation*}
    \Vert g_{\varepsilon}\Vert_{H^1} \lesssim \varepsilon^{-N}.
\end{equation*}
    \item[(ii)] A net of functions $(g_{\varepsilon})_{\varepsilon}$, is said to be $H^2$-moderate, if there exist $N\in\mathbb{N}_{0}$ such that
\begin{equation*}
    \Vert g_{\varepsilon}\Vert_{H^2} \lesssim \varepsilon^{-N}.
\end{equation*}
    \item[(iii)] A net of functions $(h_{\varepsilon})_{\varepsilon}$, is said to be ${W^{1,\infty}}$-moderate, if there exist $N\in\mathbb{N}_{0}$ such that
\begin{equation*}
    \Vert h_{\varepsilon}\Vert_{W^{1,\infty}} \lesssim \varepsilon^{-N}.
\end{equation*}
    \item[(iv)] A net of functions $(u_{\varepsilon})_{\varepsilon}$ from $C([0,T]; H^{2}(\mathbb R^d))\cap C^{1}([0,T]; H^{1}(\mathbb R^d))$ is said to be $C^1$-moderate, if there exist $N\in\mathbb{N}_{0}$ such that
\begin{equation*}
    \Vert u_{\varepsilon}(t,\cdot)\Vert \lesssim \varepsilon^{-N}
\end{equation*}
for all $t\in[0,T]$.
\end{itemize}
\end{defn}

We note that if $h_{i}\in \mathcal{E}'(\mathbb{R}^{d})$ for $i=1,...d$ and $u_0, u_1\in \mathcal{E}'(\mathbb{R}^{d})$, then the regularisations $(h_{i,\varepsilon})_{\varepsilon}$ for $i=1,...d$ of the coefficients and $(u_{0,\varepsilon})_{\varepsilon}$,  $(u_{1,\varepsilon})_{\varepsilon}$ of the Cauchy data, are moderate in the sense of the last definition.

\begin{defn}[Very weak solution]
The net $(u_{\varepsilon})_{\varepsilon}\in C([0,T]; H^{1}(\mathbb R^d))\cap C^{1}([0,T]; L^{2}(\mathbb R^d))$ is said to be a very weak solution to the Cauchy problem (\ref{Equation singular}), if there exist
\leavevmode
\begin{itemize}
    \item ${W^{1,\infty}}$-moderate regularisations of the coefficients $h_i$, for $i=1,...d$,
    \item $H^2$-moderate regularisation of $u_0$,
    \item $H^1$-moderate regularisation of $u_1$,
\end{itemize}
such that $(u_{\varepsilon})_{\varepsilon}$ solves the regularised problem
\begin{equation}
    \left\lbrace
    \begin{array}{l}
    \partial_{t}^{2}u_{\varepsilon}(t,x) - \sum_{j=1}^{d}\partial_{x_{j}}\left(h_{j,\varepsilon}(x)\partial_{x_{j}}u_{\varepsilon}(t,x)\right)=0, \,\,\,(t,x)\in\left[0,T\right]\times \mathbb{R}^{d},\\
    u_{\varepsilon}(0,x)=u_{0,\varepsilon}(x), \,\,\, \partial_{t}u_{\varepsilon}(0,x)=u_{1,\varepsilon}(x), \,\,\, x\in\mathbb{R}^{d}, \label{Equation regularized}
    \end{array}
    \right.
\end{equation}
for all $\varepsilon\in\left(0,1\right]$, and is $C^1$-moderate.
\end{defn}

\begin{thm}[Existence]
Let the coefficients $(h_i)$ be positive in the sense that all regularisations $(h_i)_\varepsilon$ are positive, for $i=1,...,d$,  and assume that the regularisations of $h_i$, $u_0$, $u_1$ satisfy the assumptions (\ref{assum coeff}) and (\ref{assum data}). Then the Cauchy problem (\ref{Equation singular}) has a very weak solution.
\end{thm}

\begin{proof}
The nets $(h_{i,\varepsilon})_{\varepsilon}$, for $i=1,...,d$ and $(u_{0,\varepsilon})_{\varepsilon}$, $(u_{1,\varepsilon})_{\varepsilon}$ are moderate by assumption. To prove the existence of a very weak solution, it remains to prove that the net $(u_{\varepsilon})_{\varepsilon}$, solution to the regularised Cauchy problem (\ref{Equation regularized}), is $C^1$-moderate. Using the estimates (\ref{Energy estimate}), (\ref{Energy estimate 1}) and the moderateness assumptions (\ref{assum coeff}) and (\ref{assum data}), we arrive at
\begin{equation*}
    \Vert u_{\varepsilon}(t,\cdot)\Vert \lesssim \varepsilon^{-2N_0 - \max \{N_1, N_2\}},
\end{equation*}
for all $t\in[0,T]$. This concludes the proof.
\end{proof}

In the next sections, we want to prove uniqueness of the very weak solution to the Cauchy problem (\ref{Equation singular}) and its consistency with the classical solution when the latter exists.

\subsection{Uniqueness}
Let us assume that we are in the case when very weak solutions to the Cauchy problem (\ref{Equation singular}) exist.
\begin{defn}[Uniqueness]
We say that the Cauchy problem (\ref{Equation singular}), has a unique very weak solution, if for all families of regularisations $(h_{i,\varepsilon})_{\varepsilon}$, $(\Tilde{h}_{i\varepsilon})_{\varepsilon}$, $(u_{0,\varepsilon})_{\varepsilon}$, $(\Tilde{u}_{0,\varepsilon})_{\varepsilon}$ and $(u_{1,\varepsilon})_{\varepsilon}$, $(\Tilde{u}_{1,\varepsilon})_{\varepsilon}$ of the coefficients $h_i$, for $i=1,...d$ and the Cauchy data $u_0$, $u_1$, satisfying
\begin{equation*}
    \Vert h_{i,\varepsilon}-\Tilde{h}_{i,\varepsilon}\Vert_{W^{1,\infty}}\leq C_{k}\varepsilon^{k} \text{\,\,for all\,\,} k>0,
\end{equation*}
\begin{equation*}
    \Vert u_{0,\varepsilon}-\Tilde{u}_{0,\varepsilon}\Vert_{H^{1}}\leq C_{m}\varepsilon^{m} \text{\,\,for all\,\,} m>0
\end{equation*}
and
\begin{equation*}
    \Vert u_{1,\varepsilon}-\Tilde{u}_{1,\varepsilon}\Vert_{L^{2}}\leq C_{n}\varepsilon^{n} \text{\,\,for all\,\,} n>0,
\end{equation*}
we have
\begin{equation*}
    \Vert u_{\varepsilon}(t,\cdot)-\Tilde{u}_{\varepsilon}(t,\cdot)\Vert_{L^{2}} \leq C_{N}\varepsilon^{N}
\end{equation*}
for all $N>0$,
where $(u_{\varepsilon})_{\varepsilon}$ and $(\Tilde{u}_{\varepsilon})_{\varepsilon}$ are the families of solutions to the related regularised Cauchy problems.
\end{defn}

\begin{thm}[Uniqueness] \label{thm uniqueness}
Let $T>0$. Suppose that $h_i(x)=h(x)$ for all $i=1, \dots, d$. Assume that for $i=1,...d$, the regularisations of the coefficients $h_i$ and the regularisations of the Cauchy data $u_0$ and $u_1$ satisfy the assumptions (\ref{assum coeff}) and (\ref{assum data}). Then, the very weak solution to the Cauchy problem (\ref{Equation singular}) is unique.
\end{thm}

\begin{proof}
Let $(h_{i,\varepsilon}, u_{0,\varepsilon}, u_{1,\varepsilon})_{\varepsilon}$, $(\Tilde{h}_{i\varepsilon}, \Tilde{u}_{0,\varepsilon}, \Tilde{u}_{1,\varepsilon})_{\varepsilon}$ be regularisations of the coefficients $h_i$, for $i=1,...d$ and the Cauchy data $u_0$, $u_1$, and let assume that they satisfy
\begin{equation*}
    \Vert h_{i,\varepsilon}-\Tilde{h}_{i,\varepsilon}\Vert_{W^{1,\infty}}\leq C_{k}\varepsilon^{k} \text{\,\,for all\,\,} k>0,
\end{equation*}
\begin{equation*}
    \Vert u_{0,\varepsilon}-\Tilde{u}_{0,\varepsilon}\Vert_{H^{1}}\leq C_{m}\varepsilon^{m} \text{\,\,for all\,\,} m>0,
\end{equation*}
and
\begin{equation*}
    \Vert u_{1,\varepsilon}-\Tilde{u}_{1,\varepsilon}\Vert_{L^{2}}\leq C_{n}\varepsilon^{n} \text{\,\,for all\,\,} n>0.
\end{equation*}
Let us denote by $U_{\varepsilon}(t,x):=u_{\varepsilon}(t,x)-\Tilde{u}_{\varepsilon}(t,x)$, where $(u_{\varepsilon})_{\varepsilon}$ and $(\Tilde{u}_{\varepsilon})_{\varepsilon}$ are the solutions to the families of regularised Cauchy problems, related to the families $(h_{i,\varepsilon}, u_{0,\varepsilon}, u_{1,\varepsilon})_{\varepsilon}$ and $(\Tilde{h}_{i\varepsilon}, \Tilde{u}_{0,\varepsilon}, \Tilde{u}_{1,\varepsilon})_{\varepsilon}$. Easy calculations show that $U_{\varepsilon}$ solves the Cauchy problem
\begin{equation}
    \left\lbrace
    \begin{array}{l}
    \partial_{t}^{2}U_{\varepsilon}(t,x) - \sum_{j=1}^{d}\partial_{x_{j}}\left(\Tilde{h}_{j,\varepsilon}(x)\partial_{x_{j}}U_{\varepsilon}(t,x)\right) = f_{\varepsilon}(t,x), \,\,\,(t,x)\in\left[0,T\right]\times \mathbb{R}^{d},\\
    U_{\varepsilon}(0,x)=(u_{0,\varepsilon}-\Tilde{u}_{0,\varepsilon})(x), \,\,\, \partial_{t}U_{\varepsilon}(0,x)=(u_{1,\varepsilon}-\Tilde{u}_{1,\varepsilon})(x), \,\,\, x\in\mathbb{R}^{d},
    \end{array}
    \right.
\end{equation}
where
\begin{equation}
    f_{\varepsilon}(t,x) = \sum_{j=1}^{d}\partial_{x_{j}}\left[\left(h_{j,\varepsilon}(x)-\Tilde{h}_{j,\varepsilon}(x)\right)\partial_{x_{j}}u_{\varepsilon}(t,x)\right].
\end{equation}
By Duhamel's principle (see, e.g. \cite{ER18}), we obtain the following representation
\begin{equation}
    U_{\varepsilon}(t, x)=V_{\varepsilon}(t, x) + \int_{0}^{t}W_{\varepsilon}(x,t-s;s)ds, \label{Representation U_eps}
\end{equation}
for $U_{\varepsilon}$, where $V_{\varepsilon}(t, x)$ is the solution to the homogeneous problem
\begin{equation}
    \left\lbrace
    \begin{array}{l}
    \partial_{t}^{2}V_{\varepsilon}(t,x) - \sum_{j=1}^{d}\partial_{x_{j}}\left(\Tilde{h}_{j,\varepsilon}(x)\partial_{x_{j}}V_{\varepsilon}(t,x)\right) = 0, \,\,\,(t,x)\in\left[0,T\right]\times \mathbb{R}^{d},\\
    V_{\varepsilon}(0,x)=(u_{0,\varepsilon}-\Tilde{u}_{0,\varepsilon})(x), \,\,\, \partial_{t}V_{\varepsilon}(0,x)=(u_{1,\varepsilon}-\Tilde{u}_{1,\varepsilon})(x), \,\,\, x\in\mathbb{R}^{d},
    \end{array}
    \right.
\end{equation}
and $W_{\varepsilon}(x,t;s)$ solves
\begin{equation}
\left\lbrace
    \begin{array}{l}
    \partial_{t}^{2}W_{\varepsilon}(x, t; s) - \sum_{j=1}^{d}\partial_{x_{j}}\left(\Tilde{h}_{j,\varepsilon}(x)\partial_{x_{j}}W_{\varepsilon}(x, t; s)\right) = 0, \,\,\,(t,x)\in\left[0,T\right]\times \mathbb{R}^{d},\\
    W_{\varepsilon}(x, 0; s)=0, \,\,\, \partial_{t}W_{\varepsilon}(x, 0; s)=f_{\varepsilon}(s,x), \,\,\, x\in\mathbb{R}^{d}.
    \end{array}
    \right.
\end{equation}
Taking the $L^2$ norm on both sides in (\ref{Representation U_eps}) and using (\ref{Energy estimate}) to estimate $V_{\varepsilon}$ and $W_{\varepsilon}$, we obtain
\begin{align}
    \Vert U_{\varepsilon}(\cdot,t)\Vert_{L^2} & \leq \Vert V_{\varepsilon}(\cdot,t)\Vert_{L^2} + \int_{0}^{T}\Vert W_{\varepsilon}(\cdot,t-s;s)\Vert_{L^2} ds \nonumber\\
    & \lesssim \left(1 + \sum_{j=1}^{d}\Vert \Tilde{h}_{j,\varepsilon}\Vert_{L^{\infty}}^{\frac{1}{2}}\right) \left[ \Vert u_{0,\varepsilon}-\Tilde{u}_{0,\varepsilon}\Vert_{H^1} + \Vert u_{1,\varepsilon}-\Tilde{u}_{1,\varepsilon}\Vert_{L^2} + \int_{0}^{T}\Vert f_{\varepsilon}(s,\cdot)\Vert_{L^2} ds\right]. \label{Estimate U_eps}
\end{align}
Let us estimate $\Vert f_{\varepsilon}(s,\cdot)\Vert_{L^2}$. We have
\begin{align*}
    \Vert f_{\varepsilon}(s,\cdot)\Vert_{L^2} & \leq \sum_{j=1}^{d}\Vert\partial_{x_{j}}\left[\left(h_{j,\varepsilon}(\cdot)-\Tilde{h}_{j,\varepsilon}(\cdot)\right)\partial_{x_{j}}u_{\varepsilon}(s,\cdot)\right]\Vert_{L^2}\\
    & \leq \sum_{j=1}^{d}\left[ \Vert\partial_{x_{j}} h_{j,\varepsilon}-\partial_{x_{j}}\Tilde{h}_{j,\varepsilon}\Vert_{L^{\infty}} \Vert\partial_{x_j}u_{\varepsilon}\Vert_{L^2} + \Vert h_{j,\varepsilon}-\Tilde{h}_{j,\varepsilon}\Vert_{L^{\infty}} \Vert\partial_{x_j}^{2}u_{\varepsilon}\Vert_{L^2}\right].
\end{align*}
In the last inequality, we used the product rule for derivatives and the fact that $\Vert\partial_{x_{j}}\left(h_{j,\varepsilon}-\Tilde{h}_{j,\varepsilon}\right)\partial_{x_j}u_{\varepsilon}\Vert_{L^2}$ and $\Vert\left(h_{j,\varepsilon}-\Tilde{h}_{j,\varepsilon}\right)\partial_{x_j}^{2}u_{\varepsilon}\Vert_{L^2}$ can be estimated by $\Vert\partial_{x_{j}} h_{j,\varepsilon}-\partial_{x_{j}}\Tilde{h}_{j,\varepsilon}\Vert_{L^{\infty}} \Vert\partial_{x_j}u_{\varepsilon}\Vert_{L^2}$ and $\Vert h_{j,\varepsilon}-\Tilde{h}_{j,\varepsilon}\Vert_{L^{\infty}} \Vert\partial_{x_j}^{2}u_{\varepsilon}\Vert_{L^2}$, respectively. We have by assumption that for all $i=1,...,d$, the net $(\Tilde{h}_{i,\varepsilon})_{\varepsilon}$ is moderate. The net $(u_{\varepsilon})_{\varepsilon}$ is also moderate as a very weak solution. Thus, there exists $N\in \mathbb{N}$ such that
\begin{equation}
    \sum_{j=1}^{d}\Vert \Tilde{h}_{j,\varepsilon}\Vert_{L^{\infty}}^{\frac{1}{2}} \lesssim \varepsilon^{-N},
\end{equation}
\begin{equation}
    \sum_{j=1}^{d}\Vert\partial_{x_j}u_{\varepsilon}\Vert_{L^2} \lesssim \varepsilon^{-N}\,\,\text{and}\,\,\Vert\Delta u_{\varepsilon}\Vert_{L^2} \lesssim \varepsilon^{-N}.
\end{equation}
On the other hand, we have that
\begin{equation*}
    \text{For}\,i=1,...,d,\,\,\Vert h_{i,\varepsilon}-\Tilde{h}_{i,\varepsilon}\Vert_{W^{1,\infty}}\leq C_{k}\varepsilon^{k} \text{\,\,for all\,\,} k>0,
\end{equation*}
\begin{equation*}
    \Vert u_{0,\varepsilon}-\Tilde{u}_{0,\varepsilon}\Vert_{H^{1}}\leq C_{m}\varepsilon^{m} \text{\,\,for all\,\,} m>0,
\end{equation*}
and
\begin{equation*}
    \Vert u_{1,\varepsilon}-\Tilde{u}_{1,\varepsilon}\Vert_{L^{2}}\leq C_{n}\varepsilon^{n} \text{\,\,for all\,\,} n>0.
\end{equation*}
It follows that
\begin{equation}
    \Vert U_{\varepsilon}(\cdot,t)\Vert_{L^2} \lesssim \varepsilon^l,
\end{equation}
for all $l\in \mathbb{N}$.
\end{proof}

\begin{rem}
The assumption that $h_i(x)=h(x)$ for all $i=1, \dots, d$, in Theorem \ref{thm uniqueness} can be removed if we know that the solution $u(t, x)$ of the problem \eqref{Equation singular} is from the class of distributions, that is, $u(t, \cdot)\in \mathcal{E}'(\mathbb{R}^{d})$ for all $t\in[0, T]$.
\end{rem}

\subsection{Consistency}
Now, we want to prove the consistency of the very weak solution with the classical one, when the latter exists, which means that, when the coefficients and the Cauchy data are regular enough, the very weak solution converges to the classical one in an appropriate norm.

\begin{thm}[Consistency] \label{thm consistency}
Let ${\boldsymbol{h}}\in \left[W^{1,\infty}\left({\mathbb{R}^d}\right)\right]^{d}$ be positive. Assume that $u_0 \in H^{2}(\mathbb{R}^d)$ and $u_1 \in H^{1}(\mathbb{R}^d)$, and let us consider the Cauchy problem
\begin{equation}
    \left\lbrace
    \begin{array}{l}
    u_{tt}(t,x) - \sum_{j=1}^{d}\partial_{x_{j}}\left(h_{j}(x)\partial_{x_{j}}u(t,x)\right)=0, \,\,\,(t,x)\in\left[0,T\right]\times \mathbb{R}^{d},\\
    u(0,x)=u_{0}(x), \,\,\, u_{t}(0,x)=u_{1}(x), \,\,\, x\in\mathbb{R}^{d}. \label{Equation consistency}
    \end{array}
    \right.
\end{equation}
Let $(u_{\varepsilon})_{\varepsilon}$ be a very weak solution of (\ref{Equation consistency}). Then, for any regularising families $h_{j,\varepsilon}=h_{j}\ast\psi_{1, \varepsilon}$ with $j=1,...d$, $u_{0,\varepsilon}=u_{0}\ast\psi_{2, \varepsilon}$ and $u_{1,\varepsilon}=u_{1}\ast\psi_{3, \varepsilon}$ for any $\psi_{k}\in C_{0}^{\infty}$, $\psi_{k}\geq 0$, $\int\psi_{k} =1$, $k=1, 2, 3,$ the net $(u_{\varepsilon})_{\varepsilon}$ converges to the classical solution of the Cauchy problem (\ref{Equation consistency}) in $L^{2}$ as $\varepsilon \rightarrow 0$.
\end{thm}

\begin{proof}
Let $u$ be the classical solution. It solves
\begin{equation*}
    \left\lbrace
    \begin{array}{l}
    u_{tt}(t,x) - \sum_{j=1}^{d}\partial_{x_{j}}\left(h_{j}(x)\partial_{x_{j}}u(t,x)\right)=0, \,\,\,(t,x)\in\left[0,T\right]\times \mathbb{R}^{d},\\
    u(0,x)=u_{0}(x), \,\,\, u_{t}(0,x)=u_{1}(x), \,\,\, x\in\mathbb{R}^{d},
    \end{array}
    \right.
\end{equation*}
and let $(u_{\varepsilon})_{\varepsilon}$ be the very weak solution. It solves
\begin{equation*}
    \left\lbrace
    \begin{array}{l}
    \partial_{t}^{2}u_{\varepsilon}(t,x) - \sum_{j=1}^{d}\partial_{x_{j}}\left(h_{j,\varepsilon}(x)\partial_{x_{j}}u_{\varepsilon}(t,x)\right)=0, \,\,\,(t,x)\in\left[0,T\right]\times \mathbb{R}^{d},\\
    u_{\varepsilon}(0,x)=u_{0,\varepsilon}(x), \,\,\, \partial_{t}u_{\varepsilon}(0,x)=u_{1,\varepsilon}(x), \,\,\, x\in\mathbb{R}^{d}.
    \end{array}
    \right.
\end{equation*}
Let us denote by $V_{\varepsilon}(t,x):=u_{\varepsilon}(t,x)-u(t,x)$. Then $V_{\varepsilon}$ solves the problem
\begin{equation*}
    \left\lbrace
    \begin{array}{l}
    \partial_{t}^{2}V_{\varepsilon}(t,x) - \sum_{j=1}^{d}\partial_{x_{j}}\left(h_{j,\varepsilon}(x)\partial_{x_{j}}V_{\varepsilon}(t,x)\right) = \beta_{\varepsilon}(t,x), \,\,\,(t,x)\in\left[0,T\right]\times \mathbb{R}^{d},\\
    V_{\varepsilon}(0,x)=(u_{0,\varepsilon}-u_0)(x), \,\,\, \partial_{t}V_{\varepsilon}(0,x)=(u_{1,\varepsilon}-u_1)(x), \,\,\, x\in\mathbb{R}^{d},
    \end{array}
    \right.
\end{equation*}
where
\begin{equation*}
    \beta_{\varepsilon}(t,x):= \sum_{j=1}^{d}\partial_{x_{j}}\left[\left(h_{j,\varepsilon}(x)-h_{j}(x)\right)\partial_{x_{j}}u(t,x)\right].
\end{equation*}
Once again, using Duhamel's principle and similar arguments as in Theorem \ref{thm consistency}, we arrive at
\begin{equation}
    \Vert V_{\varepsilon}(\cdot,t)\Vert_{L^2} \lesssim \left(1 + \sum_{j=1}^{d}\Vert h_{j,\varepsilon}\Vert_{L^{\infty}}^{\frac{1}{2}}\right) \left[ \Vert u_{0,\varepsilon}-u_0\Vert_{H^1} + \Vert u_{1,\varepsilon}-u_1\Vert_{L^2} + \int_{0}^{T}\Vert \beta_{\varepsilon}(s,\cdot)\Vert_{L^2} ds\right],
\end{equation}
where $\beta_{\varepsilon}$ is estimated by
\begin{equation}
    \Vert \beta_{\varepsilon}(s,\cdot)\Vert_{L^2} \leq \sum_{j=1}^{d}\left[ \Vert\partial_{x_{j}} h_{j,\varepsilon}-\partial_{x_{j}}h_{j}\Vert_{L^{\infty}} \Vert\partial_{x_j}u\Vert_{L^2} + \Vert h_{j,\varepsilon}-h_{j}\Vert_{L^{\infty}} \Vert\partial_{x_j}^{2}u\Vert_{L^2}\right].
\end{equation}
Since $\Vert h_{j,\varepsilon}-h_{j}\Vert_{W^{1,\infty}} \rightarrow 0$ as $\varepsilon \rightarrow 0$ and that $u$ is a classical solution, it follows that the right hand side in the last inequality tends to $0$ as $\varepsilon\rightarrow 0$. Thus
\begin{equation}
    \Vert \beta_{\varepsilon}(s,\cdot)\Vert_{L^2} \rightarrow 0\,\,\text{as}\,\varepsilon\rightarrow 0.
\end{equation}
From the other hand, for all $j=1,...,d$ the coefficients $h_{j,\varepsilon}$ are bounded since $\boldsymbol{h}\in \left[W^{1,\infty}(\mathbb{R}^d)\right]^{d}$ and we have that
\begin{equation}
    \Vert u_{0,\varepsilon}-u_0\Vert_{H^1}\rightarrow 0,
\end{equation}
and
\begin{equation}
    \Vert u_{1,\varepsilon}-u_1\Vert_{L^2}\rightarrow 0,
\end{equation}
as $\varepsilon$ tends to $0$. It follows that $(u_{\varepsilon})_{\varepsilon}$ converges to $u$ in $L^2$.
\end{proof}

\section{Numerical Experiments}

\begin{figure}[ht!]
\begin{minipage}[h]{0.45\linewidth}
\center{\includegraphics[scale=0.35]{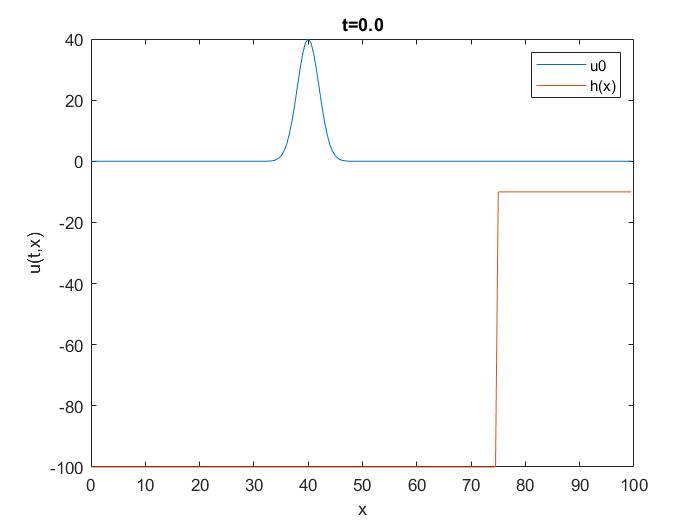}}
\end{minipage}
\hfill
\begin{minipage}[h]{0.45\linewidth}
\center{\includegraphics[scale=0.35]{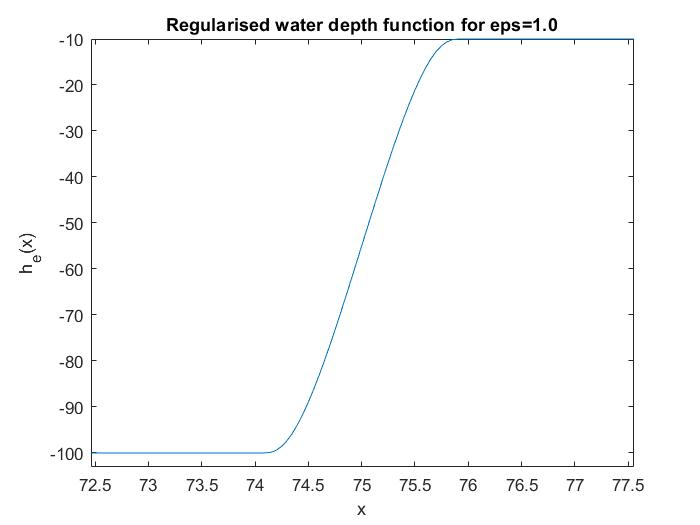}}
\end{minipage}
\caption{In the left plot, the graphics of the initial water level function $u_0(x)$ given by \eqref{u0} and of the water depth $-h_0(x)$ are drawn (coloured by blue and orange, respectively). Here, the shore is a place between $75<x\leq100$. In the right plot, for $\varepsilon=1.0$ the graphic of regularisation $h_{0, \varepsilon}(x)$ of the function $h_0(x)$ corresponding to Case 1 is given.}
\label{fig1}
\end{figure}

In this Section we carry out numerical experiments of the tsunami wave propagation in one- and two- dimensional cases. In particular, we analyse behaviours of the waves in singular topographies. Moreover, for 2D tsunami equation we develop a parallel computing algorithm to reduce the computational time. In particular, from the obtained simulations, we observe the appearance of a substantial reflected wave, travelling in the opposite direction from the point/line of singularity.

\subsection{1D case}

\begin{figure}[ht!]
\begin{minipage}[h]{0.25\linewidth}
\center{\includegraphics[scale=0.23]{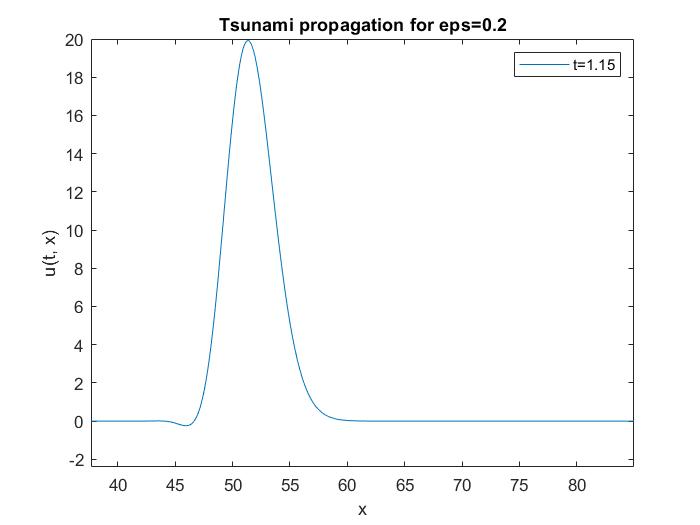}}
\end{minipage}
\hfill
\begin{minipage}[h]{0.25\linewidth}
\center{\includegraphics[scale=0.23]{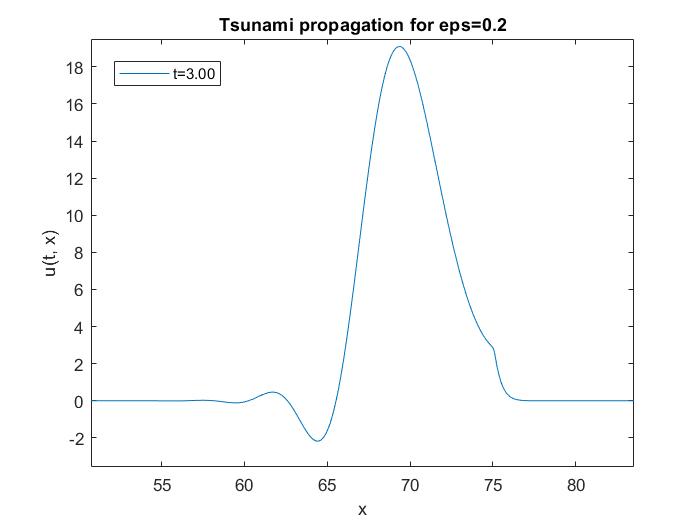}}
\end{minipage}
\hfill
\begin{minipage}[h]{0.25\linewidth}
\center{\includegraphics[scale=0.23]{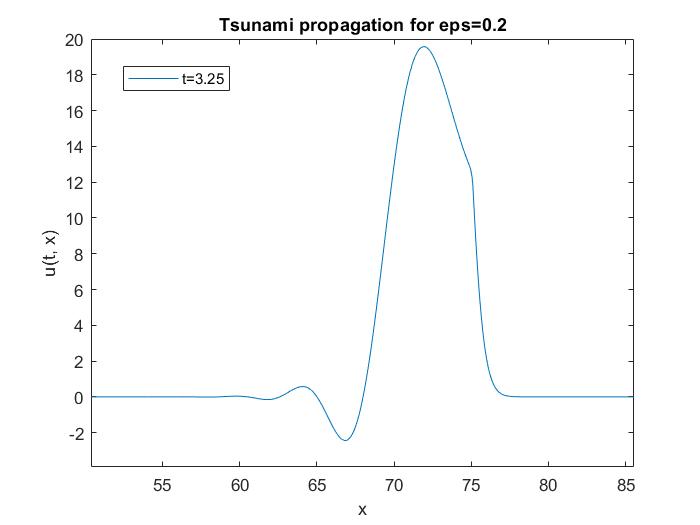}}
\end{minipage}
\hfill
\begin{minipage}[h]{0.25\linewidth}
\center{\includegraphics[scale=0.23]{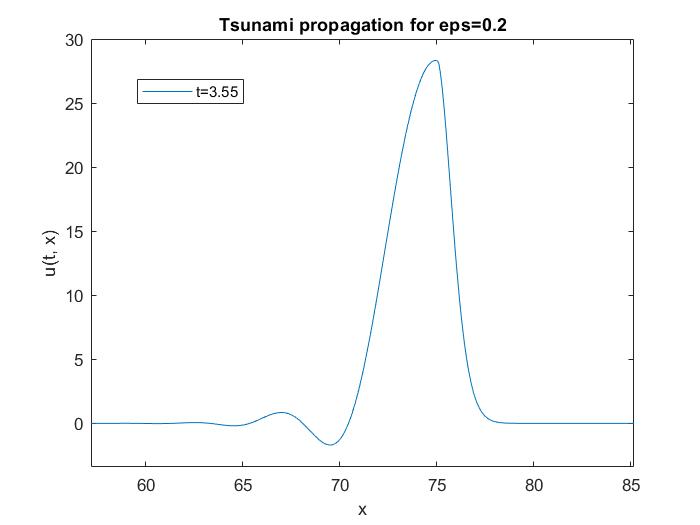}}
\end{minipage}
\hfill
\begin{minipage}[h]{0.25\linewidth}
\center{\includegraphics[scale=0.23]{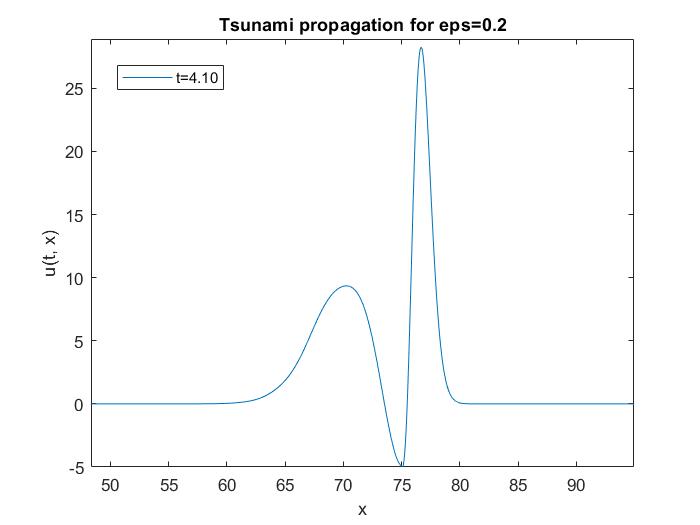}}
\end{minipage}
\hfill
\begin{minipage}[h]{0.25\linewidth}
\center{\includegraphics[scale=0.23]{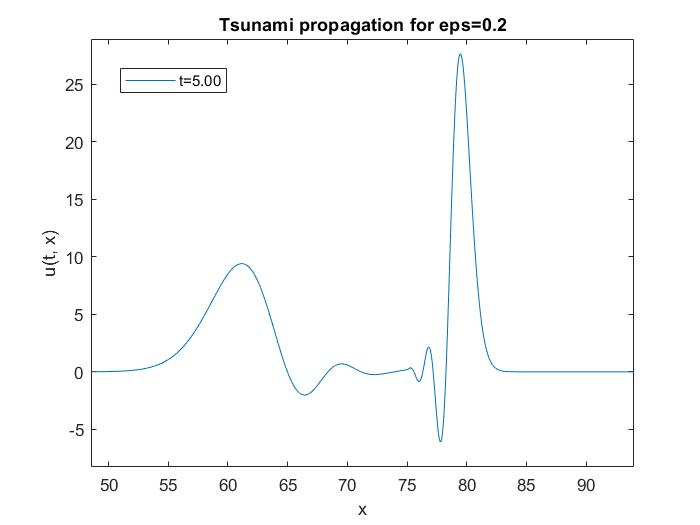}}
\end{minipage}
\caption{In these plots, an evolution of the solution of the regularised tsunami equation \eqref{num2} is given in Case 1 for $\varepsilon=0.2$ at $t=1.15, 3.00, 3.25, 3.55, 4.10, 5.00$.}
\label{fig2}
\end{figure}

\begin{figure}[ht!]
\begin{minipage}[h]{0.59\linewidth}
\center{\includegraphics[scale=0.45]{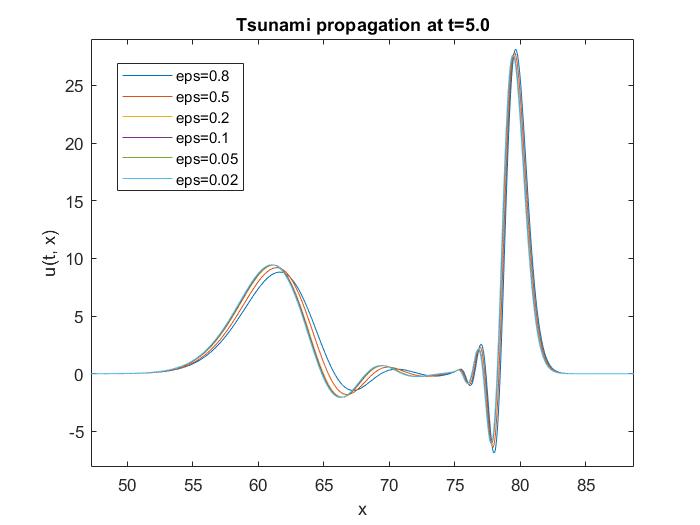}}
\end{minipage}
\caption{In this plot, the solution of the regularised tsunami equation \eqref{num2} is given in Case 1 at time $t=5.00$ for different values of the parameter $\varepsilon$, namely, for $\varepsilon=0.02, 0.05, 0.1, 0.2, 0.5, 0.8$.}
\label{fig3}
\end{figure}

\begin{figure}[ht!]
\begin{minipage}[h]{0.45\linewidth}
\center{\includegraphics[scale=0.32]{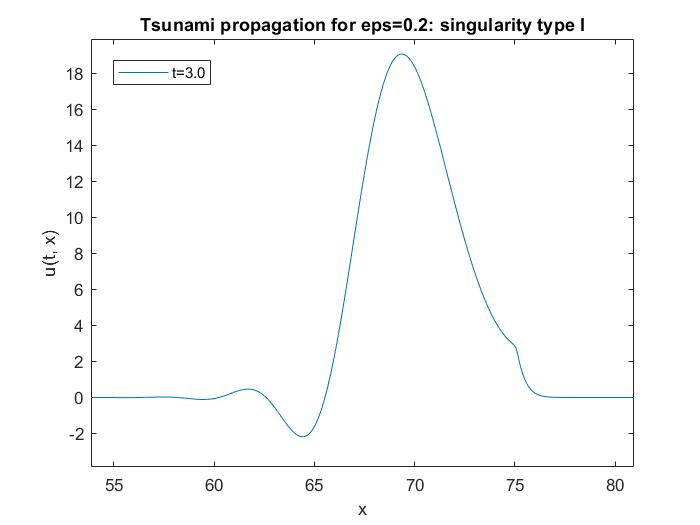}}
\end{minipage}
\hfill
\begin{minipage}[h]{0.45\linewidth}
\center{\includegraphics[scale=0.32]{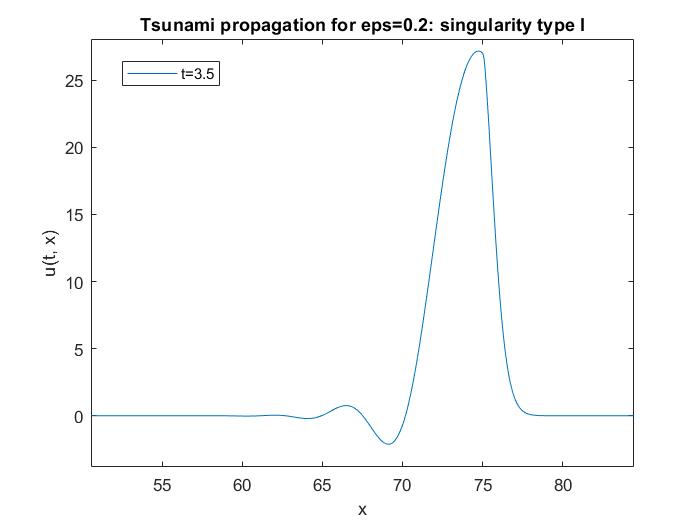}}
\end{minipage}
\hfill
\begin{minipage}[h]{0.45\linewidth}
\center{\includegraphics[scale=0.32]{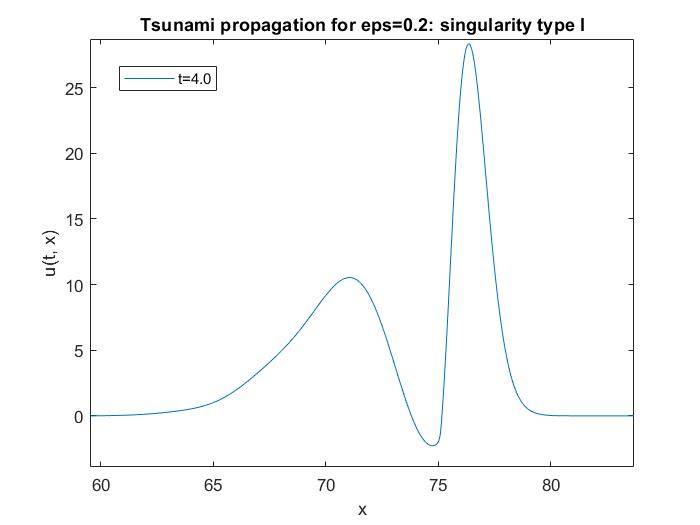}}
\end{minipage}
\hfill
\begin{minipage}[h]{0.45\linewidth}
\center{\includegraphics[scale=0.32]{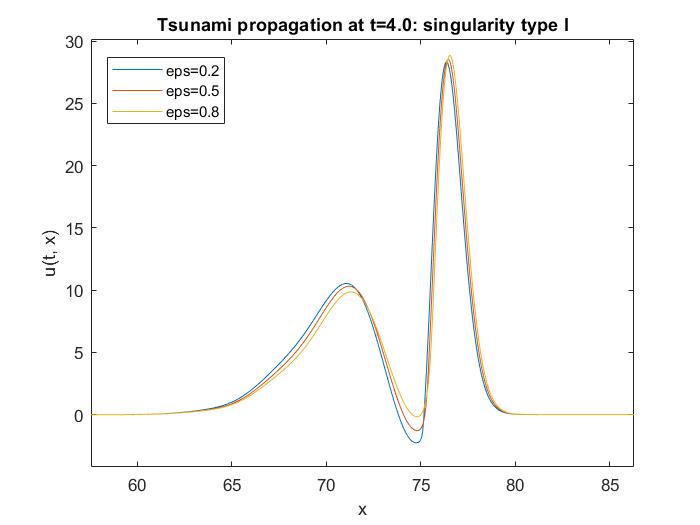}}
\end{minipage}
\caption{In these plots, the wave propagation corresponding to Case 2 is drawn at $t=3.0, 3.5, 4.0$. The right bottom plot shows that the solution of the regularised problem \eqref{num2} with the water depth function $h(x):=h_1(x)$ is stable under the changing parameter $\varepsilon$.} \label{fig4}
\end{figure}

\begin{figure}[ht!]
\begin{minipage}[h]{0.45\linewidth}
\center{\includegraphics[scale=0.32]{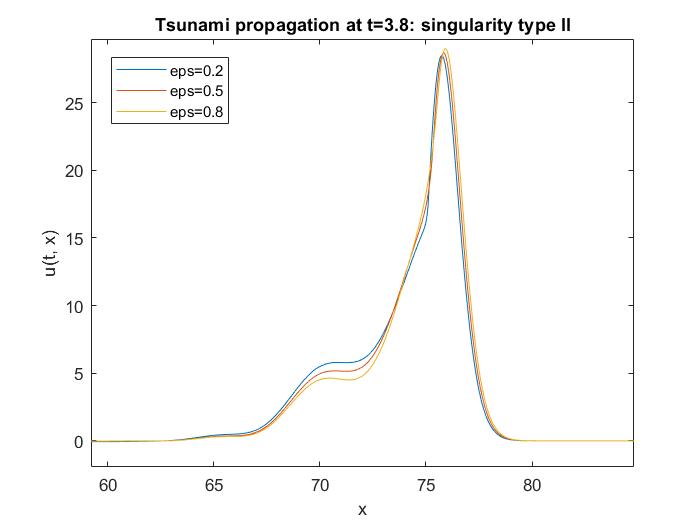}}
\end{minipage}
\hfill
\begin{minipage}[h]{0.45\linewidth}
\center{\includegraphics[scale=0.32]{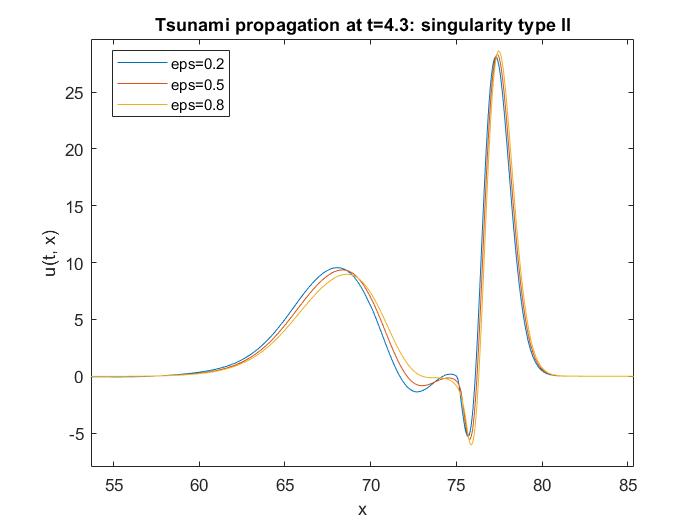}}
\end{minipage}
\hfill
\begin{minipage}[h]{0.45\linewidth}
\center{\includegraphics[scale=0.32]{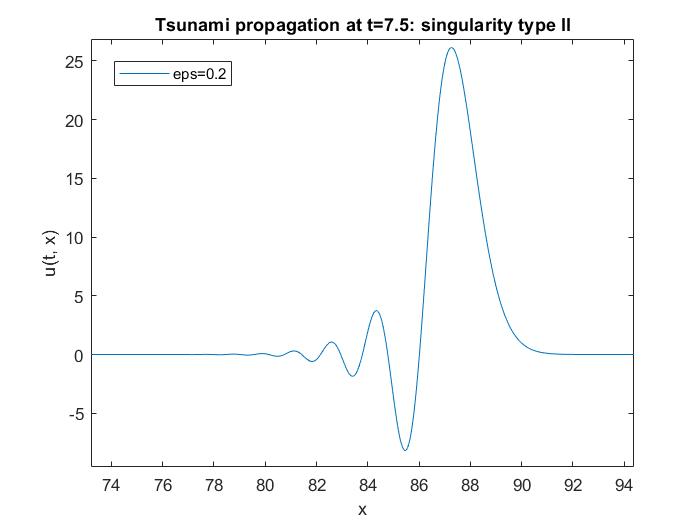}}
\end{minipage}
\hfill
\begin{minipage}[h]{0.45\linewidth}
\center{\includegraphics[scale=0.32]{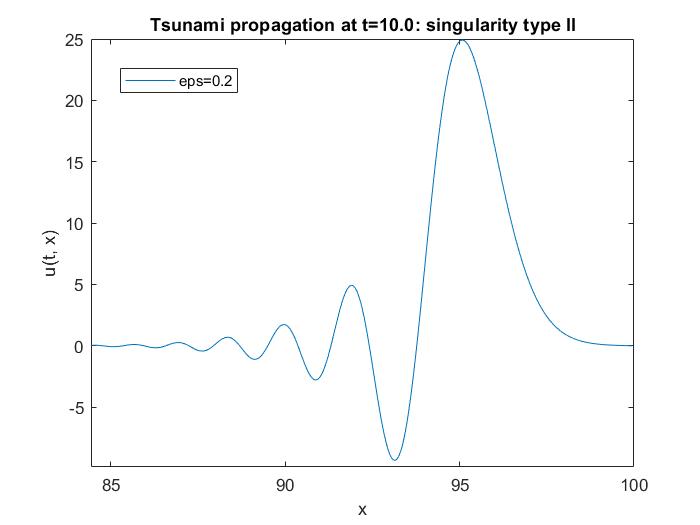}}
\end{minipage}
\caption{In these plots, the wave propagation corresponding to Case 3 is drawn at $t=3.8, 4.3$ for $\varepsilon=0.2, 0.5, 0.8$ and at $t=7.5, 10.0$ for $\varepsilon=0.2$. The plots show that the solution of the regularised problem \eqref{num2} with the water depth function $h(x):=h_2(x)$ is stable under the changing parameter $\varepsilon$.}
\label{fig5}
\end{figure}

Here, we consider 1D tsunami wave propagation equation
\begin{equation}
 u_{tt}(t,x) - \partial_{x}\left(h(x)\partial_{x}u(t,x)\right)=0, \,\,\,(t,x)\in (0,T)\times \left(0,100\right),
\label{num1}
\end{equation}
with the initial conditions
$$
 u(0,x)=u_{0}(x), \, u_{t}(0,x)=u_{1}(x),
$$
for all $x\in\left[0,100\right].$

In this work we are interested in the singular cases of the coefficient $h(x)$. Even, we can allow them to be distributional, in particular, to have $\delta$-like or $\delta^2$-like singularities. As it was theoretically outlined in \cite{RT17a} and \cite{RT17b}, we start to analyse our problem by
regularising a distributional valued function $h(x)$  by a parameter $\varepsilon$, that is, we set
\begin{equation}
h_{\varepsilon}(x):=(h*\varphi_{\varepsilon})(x)
\end{equation}
as the convolution with the mollifier
\begin{equation}
\varphi_{\varepsilon}(x)=\frac{1}{\varepsilon}\varphi(x/\varepsilon),
\end{equation}
where
$\varphi(x)=
c \exp{\left(\frac{1}{x^{2}-1}\right)}$ for $|x| < 1,$ and $\varphi(x)=0$ otherwise. Here  $c \simeq 2.2523$ to get
$\int\limits_{-\infty}^{\infty}  \varphi(x)dx=1.$

First, we study the following model situation:
\begin{itemize}
\item Case 1, when the water depth function $h(x)$ is given by
\begin{equation}
h_0(x)=
\left\lbrace
\begin{array}{l}
100, \,\,\, 0\leq x<75,\\
10, \,\,\, 75 \leq x \leq 100.
\label{h2case}
\end{array}
\right.
\end{equation}
\end{itemize}

As the second step, we study singular situations:
\begin{itemize}
\item Case 2, when the water depth function $h(x)$ has a singularity. That is,
$$
h_1(x):=h_0(x)+\delta(x-70),
$$
where $\delta$ is Dirac's function. By regularisation process described in above, we get
$$
h_{1, \varepsilon}(x)=h_{0, \varepsilon}(x) + \varphi_{\varepsilon}(x-70).
$$
\item Case 3, when the water depth function $h(x)$ has even more higher order of singularity, namely,
$$
h_2(x)=h_0(x)+\delta^{2}(x-70),
$$
in the sense that
$$
h_{2, \varepsilon}(x):=h_{0, \varepsilon}(x) + \varphi_{\varepsilon}^{2}(x-70).
$$
\end{itemize}
In what follows, we investigate all three cases.

As it was adjusted in the theoretical part, instead of \eqref{num1} we study the following
regularised problem
\begin{equation}
\partial^2 _{tt} u_{\varepsilon}(t,x) - \partial_{x}\left(h_{\varepsilon}(x)\partial_{x}u_{\varepsilon}(t,x)\right)=0, \,\,\,(t,x)\in\left(0,T\right)\times \left(0,100\right),
\label{num2}
\end{equation}
with the Cauchy data
$$
u_{\varepsilon}(0,x)=u_{0}(x), \, \partial_{t} u_{\varepsilon}(0,x)=u_{1}(x),
$$
for all $x\in\left[0,100\right].$

For our tests, we take $u_1(x) \equiv 0 $ and
\begin{equation}
\label{u0}
u_0(x)=40\exp({-(x-40)^2/8}).
\end{equation}

Now, let us analyse the results of the numerical simulations. Figure \ref{fig1} shows the graphics of the initial water level function $u_0(x)$ and the depth function $h(x):=h_0(x)$. In particular, for $\varepsilon=1.0$ the graphic of regularisation $h_{0, \varepsilon}(x)$ of the function $h_0(x)$ corresponding to Case 1 is also given. The function $h_0(x)$ has discontinuity at point $75$.

In Figure \ref{fig2} for Case 1 we study an evolution of the solution of the regularised tsunami equation \eqref{num2} for $\varepsilon=0.2$ at $t=1.15, 3.00, 3.25, 3.55, 4.10, 5.00$. From the pictures we observe that the height of the wave is starting to increase as reaching the discontinuity point. Also, a reflected wave appears.

In Figure \ref{fig3} we compare the solution of the regularised tsunami equation \eqref{num2} for $\varepsilon=0.02, 0.05,$ $0.1, 0.2, 0.5, 0.8$ at time $t=5.00$ in Case 1. From the plot we can see that the solution $u_{\varepsilon}(t, x)$ of the regularised problem \eqref{num2} is stable as $\varepsilon\to0$.

Figures \ref{fig4} and \ref{fig5} illustrate the wave propagation corresponding to singular Cases 2 and 3 at different times. The plots show that the solutions of the regularised problem \eqref{num2} with the water depth functions $h(x):=h_1(x)$ and $h(x):=h_2(x)$ are stable under the changing parameter $\varepsilon$.

In 1D case, for numerical computations we use the Crank-Nicolson method. All simulations are made in Math Lab 2018b. For all simulations we take $\Delta t = 0.05, \Delta x = 0.005.$



\subsection{Limiting behaviour as  $\varepsilon\to 0$ }

As we see from the graphs, it appears that the regularised solutions may have a limit as $\varepsilon\to 0$.

\subsubsection{Discontinuous case}
For illustration of this limiting behaviour as $\varepsilon\to 0$ of the solution of the regularised problems, as an example, we investigate Case 1 in more details. First of all, let us fix moments of $\varepsilon$ at $\varepsilon_1$ and $\varepsilon_2$. So, we will study the difference of the solution of the equation \eqref{num2} with the initial data as in \eqref{u0} at these two moments of $\varepsilon$, namely, $\Vert u_{\varepsilon_1}(t, \cdot)-u_{\varepsilon_2}(t, \cdot) \Vert_{L^2}$, and its limit as $\varepsilon_1, \varepsilon_2 \to 0$. Indeed, we have
\begin{equation}
U_{tt}(t,x) - \partial_{x}\left(h_{\varepsilon_1}(x)\partial_{x}U(t,x)\right)=\partial_{x}\left(H(x)\partial_{x}u_{\varepsilon_{2}}(t,x)\right),
\label{num3}
\end{equation}
with the Cauchy data
$$
U(0,x)=0, \, U_{t}(0,x)=0,
$$
where $U:=[u_{\varepsilon_{1}}-u_{\varepsilon_{2}}]$ and $H:=[h_{\varepsilon_1}-h_{\varepsilon_2}]$.

Since the solution $U$ linearly depends on $H$, we start by calculating it:
\begin{equation*}
\begin{split}
H(x) & = h_{\varepsilon_1}(x) - h_{\varepsilon_2}(x)=(h*\varphi_{\varepsilon_1})(x)-(h*\varphi_{\varepsilon_2})(x) \\
& = \int\limits_{-\infty}^{\infty}  h(s) \frac{1}{\varepsilon_1} \varphi\left(\frac{x-s}{\varepsilon_1}\right) d s - \int\limits_{-\infty}^{\infty}  h(s) \frac{1}{\varepsilon_2} \varphi\left(\frac{x-s}{\varepsilon_2}\right) d s.
\end{split}
\end{equation*}
Taking into account that we are considering Case 1 and using an explicit form of $h(x)$, we get
\begin{equation*}
\begin{split}
H(x) = & 100 \left[\int\limits_{x - \varepsilon_1}^{75}  \frac{1}{\varepsilon_1} \varphi\left(\frac{x-s}{\varepsilon_1}\right) d s - \int\limits_{x - \varepsilon_2}^{75} \frac{1}{\varepsilon_2} \varphi\left(\frac{x-s}{\varepsilon_2}\right) d s  \right]\\
& + 10 \left[\int\limits_{75}^{x + \varepsilon_1}  \frac{1}{\varepsilon_1} \varphi\left(\frac{x-s}{\varepsilon_1}\right) d s - \int\limits_{75}^{x + \varepsilon_2} \frac{1}{\varepsilon_2} \varphi\left(\frac{x-s}{\varepsilon_2}\right) d s  \right] \\
= & 100 \int\limits^{\frac{x-75}{\varepsilon_2}}_{\frac{x-75}{\varepsilon_1}}  \varphi(z) d z - 10 \int\limits^{\frac{x-75}{\varepsilon_2}}_{\frac{x-75}{\varepsilon_1}}  \varphi(z) d z  = 90 \int\limits^{\frac{x-75}{\varepsilon_2}}_{\frac{x-75}{\varepsilon_1}}  \varphi(z) d z.
\end{split}
\end{equation*}
Since $\varphi(x)$ is a compactly supported function, from the above calculations it is easy to see that for the sufficiently small parameters $\varepsilon_1$ and $\varepsilon_2$ the function $H(x)$ is identically zero.

\begin{rem}
Note that if instead of $\varphi_{\varepsilon_2}$ we take another mollifier $\psi_{\varepsilon_2}$ with the same properties then we obtain
\begin{equation*}
\begin{split}
H(x) = & 100 \left[\int\limits_{x - \varepsilon_1}^{75}  \frac{1}{\varepsilon_1} \varphi\left(\frac{x-s}{\varepsilon_1}\right) d s - \int\limits_{x - \varepsilon_2}^{75} \frac{1}{\varepsilon_2} \psi\left(\frac{x-s}{\varepsilon_2}\right) d s  \right]
\\
& + 10 \left[\int\limits_{75}^{x + \varepsilon_1}  \frac{1}{\varepsilon_1} \varphi\left(\frac{x-s}{\varepsilon_1}\right) d s
- \int\limits_{75}^{x + \varepsilon_2} \frac{1}{\varepsilon_2} \psi\left(\frac{x-s}{\varepsilon_2}\right) d s  \right]
\\
= & 100 \int\limits^{1}_{\frac{x-75}{\varepsilon_1}}  \varphi(z) d z - 100 \int\limits^{1}_{\frac{x-75}{\varepsilon_2}}  \psi(z) d z
+ 10 \int\limits^{\frac{x-75}{\varepsilon_2}}_{-1}  \varphi(z) d z
- 10 \int\limits^{\frac{x-75}{\varepsilon_1}}_{-1}  \psi(z) d z
\\
= & 100 \int\limits^{\frac{x-75}{\varepsilon_2}}_{\frac{x-75}{\varepsilon_1}}  \varphi(z) d z + 100 \int\limits^{1}_{\frac{x-75}{\varepsilon_2}}  (\varphi-\psi)(z) d z
+ 10 \int\limits^{\frac{x-75}{\varepsilon_2}}_{-1}  (\varphi-\psi)(z) d z
+ 10 \int\limits^{\frac{x-75}{\varepsilon_2}}_{\frac{x-75}{\varepsilon_1}}  \psi(z) d z
\\
= & 100 \int\limits^{\frac{x-75}{\varepsilon_2}}_{\frac{x-75}{\varepsilon_1}}  \varphi(z) d z + 90 \int\limits^{1}_{\frac{x-75}{\varepsilon_2}}  (\varphi-\psi)(z) d z
+ 10 \int\limits^{\frac{x-75}{\varepsilon_2}}_{\frac{x-75}{\varepsilon_1}}  \psi(z) d z.
\end{split}
\end{equation*}
Interesting to note that the last expression is also tending to zero as $\varepsilon_1, \varepsilon_2\to0$.
\end{rem}

Thus, we conclude for the sufficiently small parameters $\varepsilon_1$ and $\varepsilon_2$ the solution $U(t, x)$ of the problem \eqref{num3} is identically zero. Finally, it shows that
$$
\Vert u_{\varepsilon_1}(t, \cdot)-u_{\varepsilon_2}(t, \cdot) \Vert_{L^2}=0,
$$
as $\varepsilon_1, \varepsilon_2 \to 0$.

Therefore, a surprising conclusion is that while, in general, the solution  of the equation
\eqref{num1} may not exist in a `classical' sense for singular $h$, the limit (as $\varepsilon\to 0$) of the very weak solution family $u_\varepsilon$ may exist. We can then talk about the {\em limiting very weak solution} of \eqref{num1}  as the limit of the family $u_\varepsilon$.

\subsubsection{Irregular case}

For illustration of this limiting behaviour as $\varepsilon\to 0$ of the solution of the regularised problems, as a second example, we investigate Case 2 in more details. First of all, let us fix moments of $\varepsilon$ at $\varepsilon_1$ and $\varepsilon_2$. So, we will study the difference of the solution of the equation \eqref{num2} with the initial data as in \eqref{u0} at these two moments of $\varepsilon$, namely, $\Vert u_{\varepsilon_1}(t, \cdot)-u_{\varepsilon_2}(t, \cdot) \Vert_{L^2}$, and its limit as $\varepsilon_1, \varepsilon_2 \to 0$. Indeed, we have
\begin{equation}
U_{tt}(t,x) - \partial_{x}\left(h_{\varepsilon_1}(x)\partial_{x}U(t,x)\right)=\partial_{x}\left(H(x)\partial_{x}u_{\varepsilon_{2}}(t,x)\right),
\label{num4}
\end{equation}
with the Cauchy data
$$
U(0,x)=0, \, U_{t}(0,x)=0,
$$
where $U:=[u_{\varepsilon_{1}}-u_{\varepsilon_{2}}]$ and $H:=[h_{\varepsilon_1}-h_{\varepsilon_2}]$.

By changing
$$
V(t, x):= \int\limits_{-\infty}^{x} U(t, s) ds,
$$
instead of the equation \eqref{num4} we get
\begin{equation}
V_{tt}(t,x) - h_{\varepsilon_1}(x)\partial_{xx}V(t,x)=H(x)\partial_{x}u_{\varepsilon_{2}}(t,x),
\label{num5}
\end{equation}
with the Cauchy data
$$
V(0,x)=0, \, V_{t}(0,x)=0.
$$

Repeating the above procedure, let us calculate $H$:
\begin{equation*}
\begin{split}
H(x) & = h_{\varepsilon_1}(x) - h_{\varepsilon_2}(x)=(h*\varphi_{\varepsilon_1})(x)-(h*\varphi_{\varepsilon_2})(x) \\
& = \int\limits_{-\infty}^{\infty}  h(s) \frac{1}{\varepsilon_1} \varphi\left(\frac{x-s}{\varepsilon_1}\right) d s - \int\limits_{-\infty}^{\infty}  h(s) \frac{1}{\varepsilon_2} \varphi\left(\frac{x-s}{\varepsilon_2}\right) d s.
\end{split}
\end{equation*}
Taking into account that we are considering Case 2 and using an explicit form of $h(x)$, we get
\begin{equation*}
H(x) =  A(x) + D(x),
\end{equation*}
where
\begin{equation*}
A(x) = 90 \int\limits^{\frac{x-75}{\varepsilon_2}}_{\frac{x-75}{\varepsilon_1}}  \varphi(z) d z \,\,\,\, \hbox{and} \,\,\,\, D(x) = \frac{1}{\varepsilon_1} \varphi\left(\frac{x-70}{\varepsilon_1}\right) - \frac{1}{\varepsilon_2} \varphi\left(\frac{x-70}{\varepsilon_2}\right).
\end{equation*}
Since $\varphi(x)$ is a compactly supported function, from the above calculations it is easy to see that for the sufficiently small parameters $\varepsilon_1$ and $\varepsilon_2$ the function $A(x)$ is identically zero. Also, we note that the function $D(x)$ has a compact support
$$
\supp D\subset [70-\max(\varepsilon_1, \varepsilon_2), 70+\max(\varepsilon_1, \varepsilon_2)].
$$
Without loss of generality, we assume that $\varepsilon_1\geq\varepsilon_2$. Then it is clear that
$$
\supp D = \supp_{sing} h_{\varepsilon_1}\subset[70-\varepsilon_1, 70+\varepsilon_1]=:\Omega_{\varepsilon_1}.
$$
Note that when $x\in\mathbb R\setminus\Omega_{\varepsilon_1}$ we have the {\it Discontinuous case}. Now we are interested in the case when $x\in\Omega_{\varepsilon_1}$. Thus, for small enough $\varepsilon_1$, we have
$$
\varepsilon_1 V_{tt}(t,x) - \left[\varepsilon_1h_{0, \varepsilon_1}(x) + \varphi\left(\frac{x-70}{\varepsilon_1}\right)\right]\partial_{xx}V(t,x) = \varepsilon_1 H(x)\partial_{x}u_{\varepsilon_{2}}(t,x),
$$
and by neglecting the small terms, we arrive at the elliptic type problem
\begin{equation}
- \varphi\left(\frac{x-70}{\varepsilon_1}\right)\partial_{xx}V(t,x)=\left(\varphi\left(\frac{x-70}{\varepsilon_1}\right) - \frac{\varepsilon_1}{\varepsilon_2} \varphi\left(\frac{x-70}{\varepsilon_2}\right)\right)\partial_{x}u_{\varepsilon_{2}}(t,x).
\label{num6}
\end{equation}
Dividing both sides of \eqref{num6} by $\varphi\left(\frac{x-70}{\varepsilon_1}\right)$, we obtain
\begin{equation}
- \partial_{xx}V(t,x)=\left(1 - \frac{\varepsilon_1}{\varepsilon_2} \hat{\varphi}(x, \varepsilon_2)\right)\partial_{x}u_{\varepsilon_{2}}(t,x),
\label{num7}
\end{equation}
where $\supp{\hat{\varphi}}=\Omega_{\varepsilon_2}$ and $x\in\Omega_{\varepsilon_1}$. By integrating over $\int_{-\infty}^{x}$ and taking into account that $U(t,x) = \partial_{x}V(t,x)$, we arrive at
\begin{equation}
U(t, x) = u_{\varepsilon_{2}}(t, 70+\varepsilon_1) - u_{\varepsilon_{2}}(t, 70 - \varepsilon_1)  + \frac{\varepsilon_1}{\varepsilon_2} \int\limits_{70-\varepsilon_2}^{\min(70+\varepsilon_2, x)}\hat{\varphi}(s, \varepsilon_2)\partial_{s}u_{\varepsilon_{2}}(t,s) d s,
\label{num7-sol}
\end{equation}
for $x\in\Omega_{\varepsilon_1}$.

Now we need to estimate \eqref{num7-sol} in $L^2$-norm. For this, by adapting the energy conservation formula \eqref{CL-01} to $u_{\varepsilon_{2}}$, we obtain
\begin{equation}
\label{CL-02}
\begin{split}
\Vert \partial_{x}u_{\varepsilon_{2}}(t,\cdot)\Vert_{L^2}^{2} \leq \frac{1}{10}\Vert h_{\varepsilon_{2}}^{\frac{1}{2}}\partial_{x}u_{0}\Vert_{L^2}^{2} & = \frac{1}{10}\int\limits_{\mathbb R} h_{\varepsilon_{2}}(s)|\partial_{x}u_{0}(s)|^{2} d s \\
& = \frac{1}{10}\int\limits_{\mathbb R} \partial_{s}\hat{h}_{\varepsilon_{2}}(s)|\partial_{s}u_{0}(s)|^{2} d s,
\end{split}
\end{equation}
where $h_{\varepsilon_{2}}(s):=\partial_{s}\hat{h}_{\varepsilon_{2}}(s)$. Integrating by parts and taking into account the properties of $u_0$, from \eqref{CL-02} we get
\begin{equation}
\label{CL-03}
\begin{split}
\Vert \partial_{x}u_{\varepsilon_{2}}(t,\cdot)\Vert_{L^2}^{2} & \leq \frac{1}{10}\int\limits_{\mathbb R} \partial_{s}\hat{h}_{\varepsilon_{2}}(s)|\partial_{s}u_{0}(s)|^{2} d s \\
& = \frac{1}{5}\int\limits_{\mathbb R} \hat{h}_{\varepsilon_{2}}(s)|\partial_{s}^{2}u_{0}(s) \partial_{s}u_{0}(s)| d s\\
& = \frac{1}{5} \|\hat{h}_{\varepsilon_{2}}\|_{L^{\infty}} \|\partial_{s}^{2}u_{0}\|_{L^2} \|\partial_{s}u_{0}\|_{L^2}.
\end{split}
\end{equation}
The term $\Vert \partial_{x}u_{\varepsilon_{2}}(t,\cdot)\Vert_{L^2}$ does not blow up as $\varepsilon_2\to0$ since $\hat{h}_{\varepsilon_{2}}\to\hat{h}\in L^{\infty}$ as $\varepsilon_2\to0$. Repeating the process for $u_{\varepsilon_{2}}$, one obtains that $u_{\varepsilon_{2}}$ is also regular in $\Omega_{\varepsilon_1}$.

Since for $x\in\mathbb R\setminus\Omega_{\varepsilon_1}$ the function $U(t, x)$ equal the solution corresponding to Case 1, and due to the fact that the volume of the domain $\Omega_{\varepsilon_1}$ tends to zero as $\varepsilon_1\to0$, we conclude that for the sufficiently small parameters $\varepsilon_1$ and $\varepsilon_2$ the solution $U(t, x)$ of the problem \eqref{num4} tends to zero. Finally, it shows that
$$
\Vert u_{\varepsilon_1}(t, \cdot)-u_{\varepsilon_2}(t, \cdot) \Vert_{L^2}\to0,
$$
as $\varepsilon_1, \varepsilon_2 \to 0$.

Therefore, a surprising conclusion is that while, in general, the solution  of the equation
\eqref{num1} may not exist in a `classical' sense for singular $h$, the limit (as $\varepsilon\to 0$) of the very weak solution family $u_\varepsilon$ may exist. We can then talk about the {\em limiting very weak solution} of \eqref{num1}  as the limit of the family $u_\varepsilon$.

\subsubsection{Tests for singularities}

\begin{figure}[ht!]
\begin{minipage}[h]{0.45\linewidth}
\center{\includegraphics[scale=0.32]{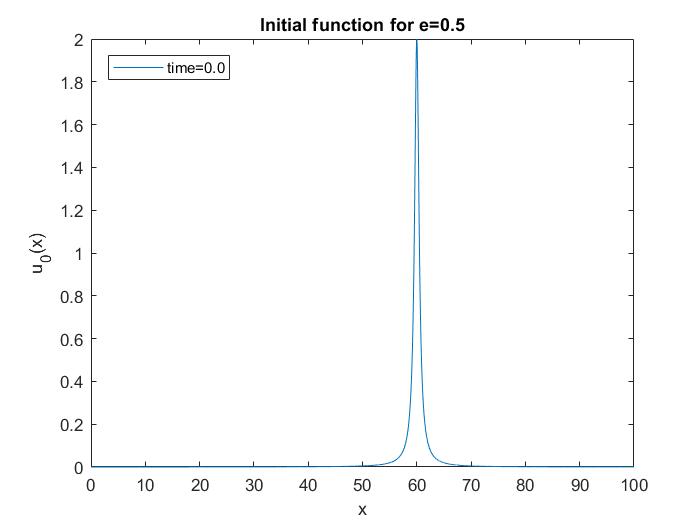}}
\end{minipage}
\hfill
\begin{minipage}[h]{0.45\linewidth}
\center{\includegraphics[scale=0.32]{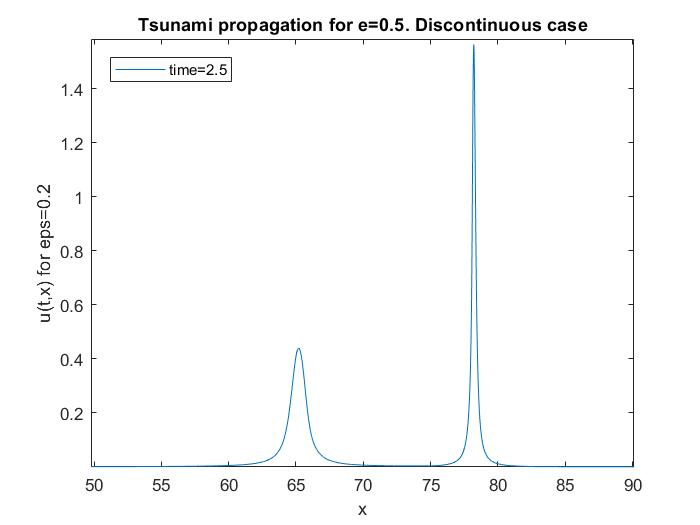}}
\end{minipage}
\hfill
\begin{minipage}[h]{0.45\linewidth}
\center{\includegraphics[scale=0.32]{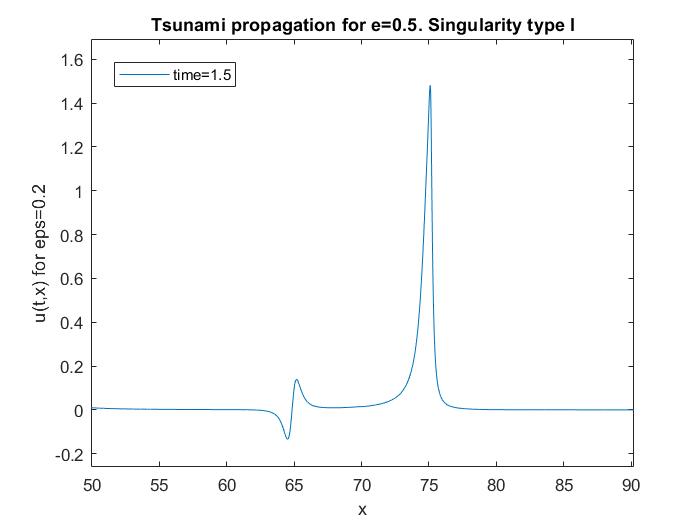}}
\end{minipage}
\hfill
\begin{minipage}[h]{0.45\linewidth}
\center{\includegraphics[scale=0.32]{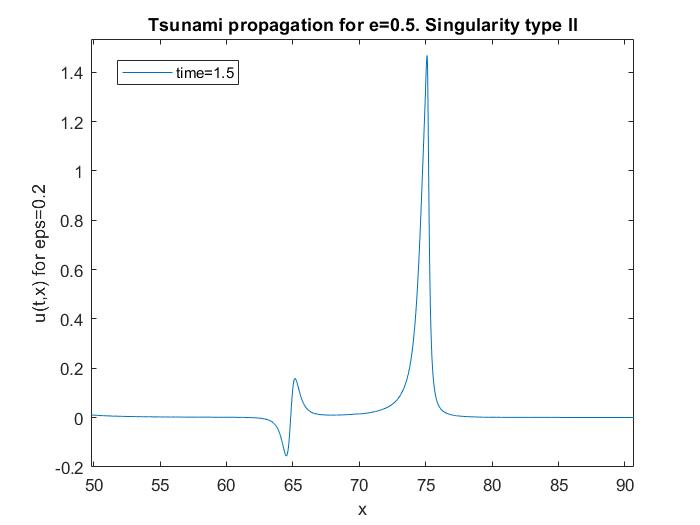}}
\end{minipage}
\caption{In these plots, the initial function $u_0$ given by \eqref{u0-2} and the wave propagation corresponding to the discontinuous and singular type I and II cases are drawn for $e=0.5$, respectively. All simulations are done for $\varepsilon=0.2$.}
\label{fig6a}
\end{figure}

\begin{figure}[ht!]
\begin{minipage}[h]{0.45\linewidth}
\center{\includegraphics[scale=0.32]{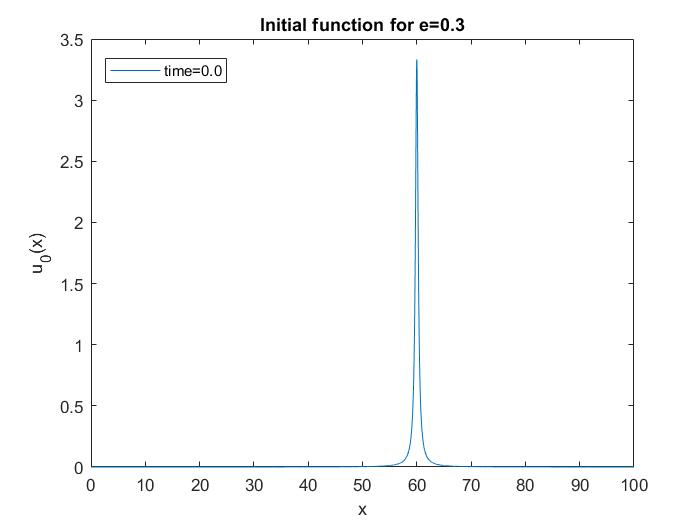}}
\end{minipage}
\hfill
\begin{minipage}[h]{0.45\linewidth}
\center{\includegraphics[scale=0.32]{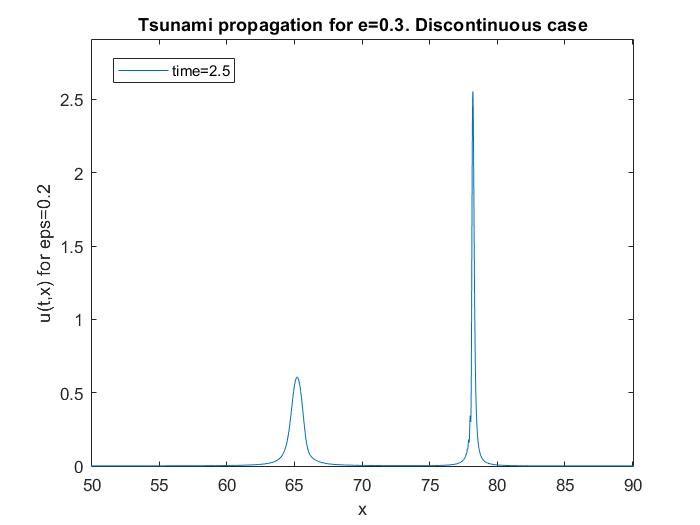}}
\end{minipage}
\hfill
\begin{minipage}[h]{0.45\linewidth}
\center{\includegraphics[scale=0.32]{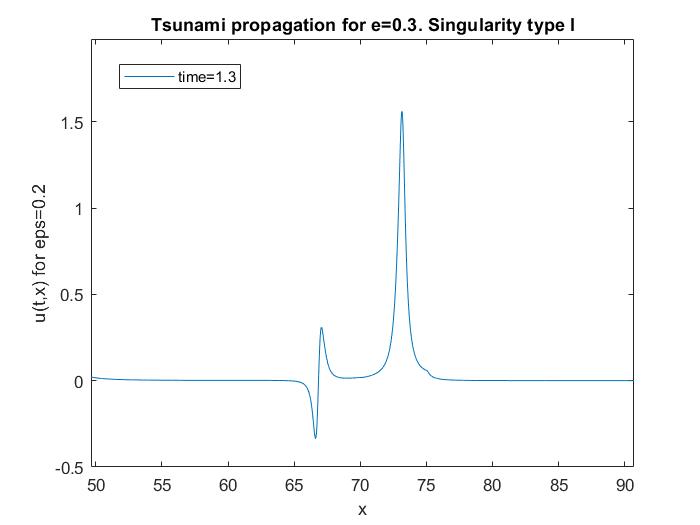}}
\end{minipage}
\hfill
\begin{minipage}[h]{0.45\linewidth}
\center{\includegraphics[scale=0.32]{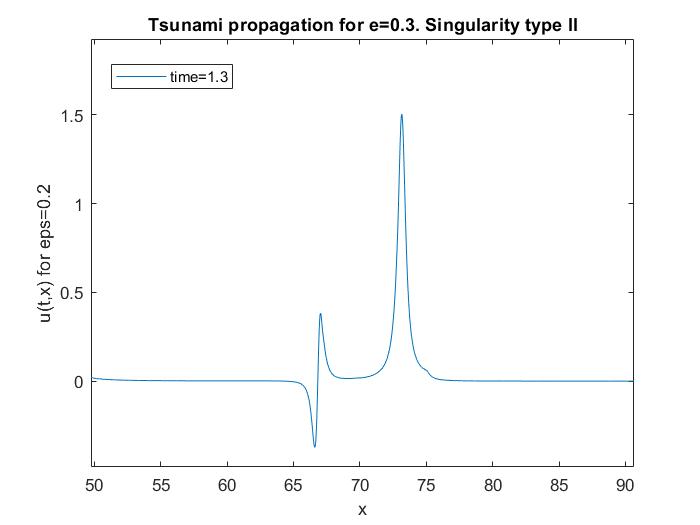}}
\end{minipage}
\caption{In these plots, the initial function $u_0$ given by \eqref{u0-2} and the wave propagation corresponding to the discontinuous and singular type I and II cases are drawn for $e=0.3$, respectively. All simulations are done for $\varepsilon=0.2$.}
\label{fig6b}
\end{figure}

\begin{figure}[ht!]
\begin{minipage}[h]{0.45\linewidth}
\center{\includegraphics[scale=0.32]{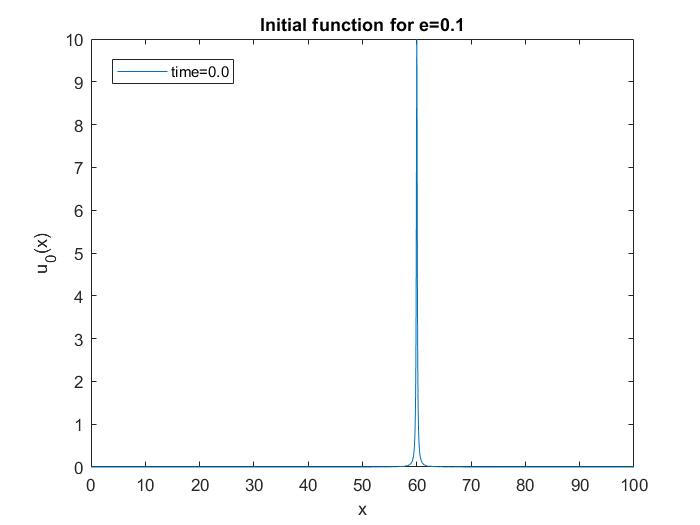}}
\end{minipage}
\hfill
\begin{minipage}[h]{0.45\linewidth}
\center{\includegraphics[scale=0.32]{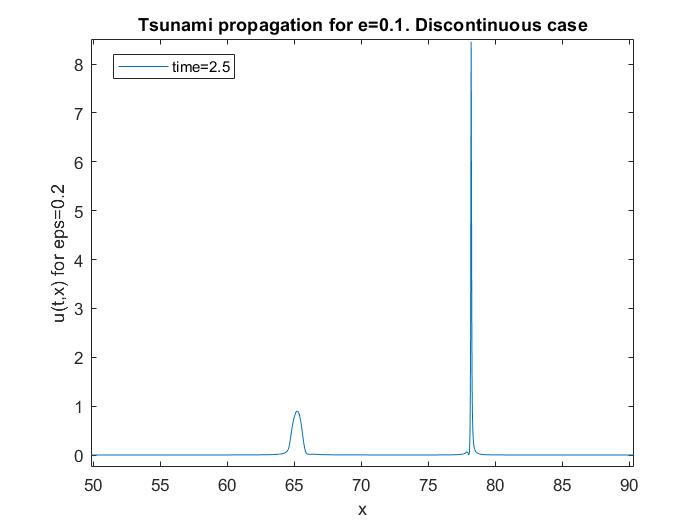}}
\end{minipage}
\hfill
\begin{minipage}[h]{0.45\linewidth}
\center{\includegraphics[scale=0.32]{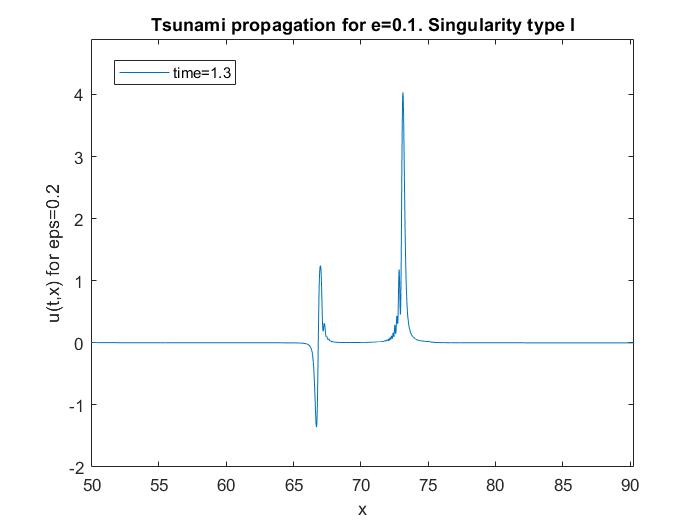}}
\end{minipage}
\hfill
\begin{minipage}[h]{0.45\linewidth}
\center{\includegraphics[scale=0.32]{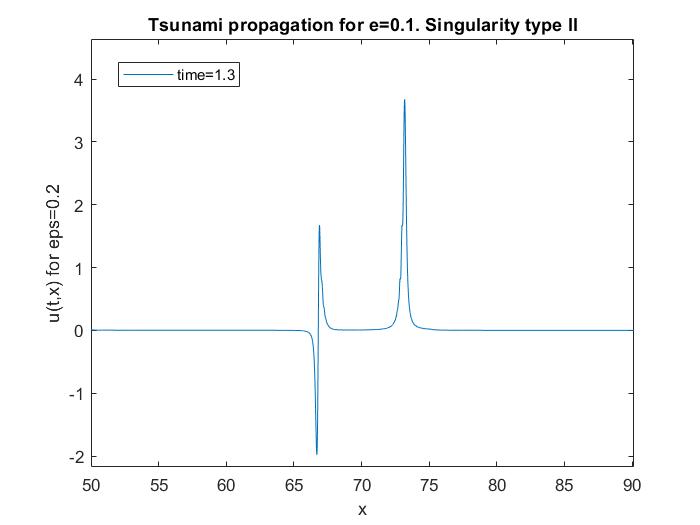}}
\end{minipage}
\caption{In these plots, the initial function $u_0$ given by \eqref{u0-2} and the wave propagation corresponding to the discontinuous and singular type I and II cases are drawn for $e=0.1$, respectively. All simulations are done for $\varepsilon=0.2$.}
\label{fig6c}
\end{figure}

To investigate singularities, we consider
\begin{itemize}
\item Discontinuous case, when the water depth function $h(x)$ is given by
\begin{equation}
h_0(x)=
\left\lbrace
\begin{array}{l}
100, \,\,\, 0\leq x<75,\\
10, \,\,\, 75 \leq x \leq 100.
\label{h2case}
\end{array}
\right.
\end{equation}
\item Singular type I case, when the water depth function $h(x)$ has a singularity. That is,
$$
h_1(x):=h_0(x)+100\delta(x-70),
$$
where $\delta$ is Dirac's function. By regularisation process described in above, we get
$$
h_{1, \varepsilon}(x)=h_{0, \varepsilon}(x) + 100\varphi_{\varepsilon}(x-70).
$$
\item Singular type II case, when the water depth function $h(x)$ has even more higher order of singularity, namely,
$$
h_2(x)=h_0(x)+100\delta^{2}(x-70),
$$
in the sense that
$$
h_{2, \varepsilon}(x):=h_{0, \varepsilon}(x) + 100\varphi_{\varepsilon}^{2}(x-70).
$$
\end{itemize}
Here, we simulate the cases of $h$ considering instead of $u_0$ given by \eqref{u0} the function
\begin{equation}
\label{u0-2}
u_0(x)=\frac{e}{(x-60)^2+e^2},
\end{equation}
for $e\in\mathbb R_{+}$.

In Figures \ref{fig6a}-\ref{fig6c} we test for $e=0.5, 0.3, 0.1.$ The reason for this investigation is to see the strength of the singularity in the reflected wave. We observe that the singularity of the reflected wave (the sharpness of the peak) is less than that of the main wave in the Case I, while the reflected singularity seems to be of the same strength in Cases II and III.
In this respect, the behaviour in Case I resembles more that of the conical refraction corresponding to multiplicities (as in \cite{KR07}), while Cases II and III appear to be more like acoustic echo-type effects for singular media (as in \cite{MRT19}).

In all cases the second wave is smaller in size. The reflected wave has only one positive component in Case I, while it has both positive and negative parts in Cases II and III.

\subsection{2D case}

\begin{figure}[ht!]
\begin{minipage}[h]{0.45\linewidth}
\center{\includegraphics[scale=0.32]{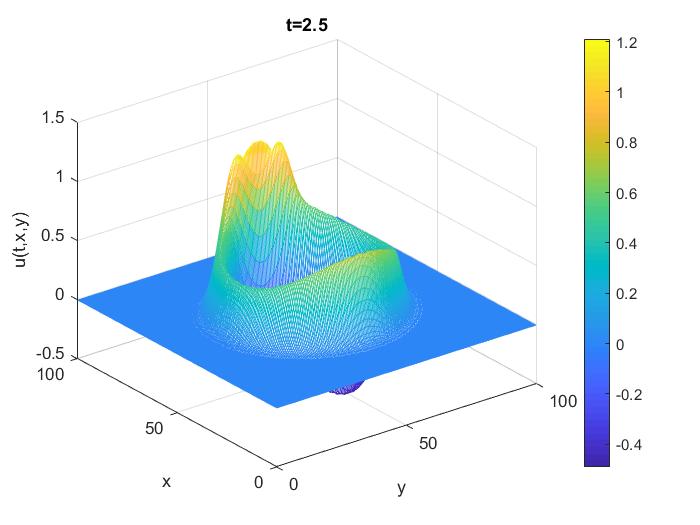}}
\end{minipage}
\hfill
\begin{minipage}[h]{0.45\linewidth}
\center{\includegraphics[scale=0.32]{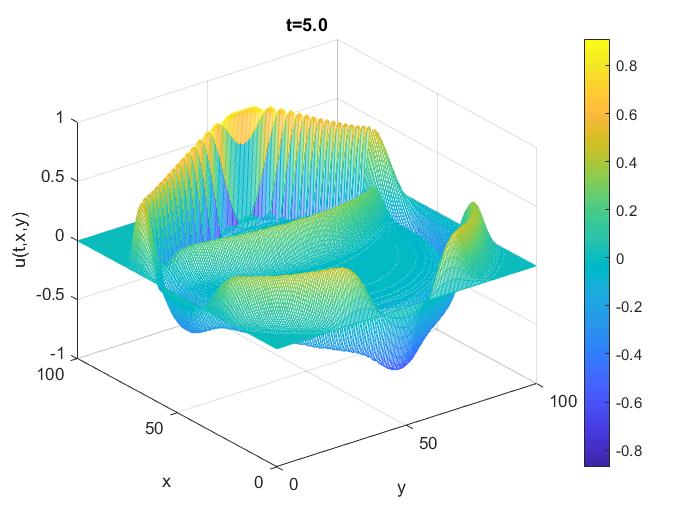}}
\end{minipage}
\caption{Displacement of the wave corresponding to the equation \eqref{2dwave_1} for $\varepsilon=0.8$
at $t=2.5$ and $t=5.0$.}
\label{fig6}
\end{figure}

In the domain $\left[0,T\right]\times \left[0,100\right]\times \left[0,100\right]$, we simulate the following boundary value problem for the 2D tsunami equation
$$
u_{tt}(t,x,y) - [\partial_{x}\left(H(x,y)\partial_{x}u(t,x,y)\right)+\partial_{y}\left(H(x,y)\partial_{y}u(t,x,y)\right)]=0,
$$
with the initial data
$$
u(0,x,y)=u_{0}(x,y), \,\,\,  u_{t}(0,x,y)=u_{1}(x,y), \,\,\,x,y\in\left[0,100\right],
$$
and boundary conditions
$$
u(t,0,y)=0,\, u(t,100,y)=0,\, u(t,x,0)=0,\, u(t,x,100)=0,
$$
for $x,y\in\left[0,100\right],$ for $t\in\left[0,T\right]$.

Now we introduce a space-time grid with steps $h_x, h_y, \tau  $ in the variables $t, x, y$, respectively:
\begin{equation}
\label{2dwave_2}
 \omega_{h_x,h_y}^\tau={\{t_k=k\tau,\,\, k=\overline{0,M}; x_i=ih_x, y_j=jh_y,\,\, i, j=\overline{0,N}}\},
 \end{equation}
where $\tau M=T, h_x N=h_y N=100$. For numerically solving this problem we use an implicit finite difference scheme \cite{Sam} and the cyclic reduction method \cite{Gode11}.

In the two-dimensional model, we consider 'Case 1' corresponding to 1D simulations. For the water depth function $H(x, y)$ we put
$$
H(x, y) := h_0(x),
$$
in $x$ variable and constant in $y$ variable. Here $h_0(x)$ is as in \eqref{h2case}. Eventually, for the regularisation of $H(x, y)$ we get $H_{\varepsilon}(x, y) = h_{0, \varepsilon}(x).$ For the simulations we solve the following regularised equation
\begin{equation}
\label{2dwave_1}
u_{tt}(t,x,y) - [\partial_{x}\left(H_{\varepsilon}(x,y)\partial_{x}u(t,x,y)\right)+\partial_{y}\left(H_{\varepsilon}(x,y)\partial_{y}u(t,x,y)\right)]=0,
\end{equation}
with the initial functions
$$
u_0(x, y) = 50\exp({-((x-40)^2+(y-50)^2)/8})
$$
and $u_1(x, y) = 0$.

In 2D case, numerical computations and simulations are made in python by using the cyclic reduction method. For all simulations we take $\Delta t = 0.5, \Delta x = 0.05, \Delta y = 0.05.$

\section{GPU computing}

Modeling a wider area and long-term modeling using a standard personal computer requires more time, and it is often very important to reduce the computation time. Modern graphical processing units provide a powerful instrument for parallel processing with massively data-parallel throughput-oriented multi-core processors capable of providing TFLOPS of computing performance and quite high memory bandwidth. So with the aim of reducing computation time in this work, we use GPU computing.

In this section we show the results obtained on a desktop computer with configuration  4352 cores GeForce RTX 2080 TI, NVIDIA GPU together with a CPU Intel Core(TM) i7-9800X, 3.80 GHz, RAM 64Gb.
Simulation parameters are configured as follows. Mesh size is uniform in both directions with $\Delta x = \Delta y$ and numerical time step $\Delta t$ is 0.05, and simulation time is $T=5.0$, therefore the total number of time steps is 100. To present more realistic data, we tested five cases with computational domain sizes of $256\times256, 512\times512, 1024\times1024, 2048\times2048$ and $4096\times4096$.\par
The performance of a parallel algorithm is determined by calculating its speedup. The speedup is defined as the ratio of the execution time of the sequential algorithm for a particular problem to the execution time of the parallel algorithm.
$$\mathrm{Speedup}=\frac{\mathrm{CPUtime}}{\mathrm{GPUtime}} $$\par
In  Table \ref{table1} we report the execution times in seconds for serial (CPU time) and CUDA (GPU time) implementation of cyclic reduction method to the problem \eqref{2dwave_1} together with the values of the speedup.
\begin{table}
\caption{Execution timing and speedup with the Intel Core(TM) i7-9800X, 3.80 GHz,  NVIDIA RTX 2080 TI}
\begin{tabular}{  c | c | c | c}
\hline
Domain sizes & CPU time  & GPU time & Speedup \\
\hline
$256\times 256$ & 0.91  & 0.88 & 1.03\\
$512\times 512$ & 3.73  & 2.07 & 1.8\\
$1024\times 1024$ & 15.92   & 7.16 & 2.22\\
$2048\times 2048$  & 64.8  & 20.30 & 3.19\\
$4096\times 4096$ & 280.54 & 62.76 & 4.47\\
\hline
\end{tabular}\label{table1}
\end{table}

\end{document}